\documentclass{amsart}
\usepackage{textgreek}
\usepackage{relsize}
\usepackage[dvipsnames]{xcolor}
\usepackage{parskip}
\usepackage{amssymb}
\usepackage{latexsym}
\usepackage{amsmath}
\usepackage{enumerate}
\usepackage{amsthm}
\usepackage{wasysym}
\usepackage{stmaryrd}
\usepackage{mathrsfs}
\usepackage{pifont}
\usepackage{tikz}
\usepackage[f]{esvect}

\newcommand{\qee} {\hspace*{2mm}\hfill \ding{109}}

\renewcommand{\iff}{\leftrightarrow}
\renewcommand{\leq}{\leqslant}
\renewcommand{\geq}{\geqslant}

\renewcommand{\phi}{\varphi}

\renewcommand{\Theta}{\varTheta}
\renewcommand{\Phi}{\varPhi}
\renewcommand{\Psi}{\varPsi}
\renewcommand{\Xi}{\varXi}
\renewcommand{\Gamma}{\varGamma}
\newcommand{\qedright}{\belowdisplayskip=-12pt}

\newtheorem{theorem}{Theorem}[section]
\newtheorem{prop}[theorem]{Proposition}
\newtheorem{define}[theorem]{Definition}
\newenvironment{definition}{\begin{define} \rm}{\qee\end{define}}
\newtheorem{exa}[theorem]{Example}

\newtheorem{exerc}[theorem]{Exercise}

\newtheorem{conj}[theorem]{Conjecture}

\newtheorem{ques}[theorem]{Open Question}
\newenvironment{question}{\begin{ques} \rm}{\qee\end{ques}}
\newtheorem{jques}[theorem]{Question}
\newenvironment{jquestion}{\begin{jques} \rm}{\qee\end{jques}}
\newtheorem{lem}[theorem]{Lemma}
\newtheorem{lemma}{Lemma}[section]
\newtheorem{cor}[theorem]{Corollary}

\newtheorem{rem}[theorem]{Remark}

\newcommand{\mc}[1]{\mathcal #1}
 \newcommand{\tupel}[1]{{\langle #1 \rangle}}
\newcommand{\verz}[1]{\{ #1 \}}
\newcommand{\To}{\Rightarrow}
\newcommand{\Iff}{\Leftrightarrow}
\newcommand{\bles}{\mathbin{<}}
\newcommand{\bleq}{\mathbin{\leq}}
\newcommand{\bin}{\mathbin{\in}}
\newcommand{\num}[1]{{\underline {#1}}}
\newcommand{\mf}[1]{{\mathfrak {#1}}}
\newcommand{\hyph}{\mbox{-}}

\newcommand{\pam}{\ensuremath{{\sf PA}^{-}}}
\newcommand{\pa}{{\sf PA}}

\newcommand{\belowz}{\sqsubset}
\newcommand{\below}[2]{#1\sqsubset #2}
\newcommand{\bbelow}[2]{#1\mathop{\sqsubset}#2}

\newcommand{\X}{{\sf X}}
\newcommand{\odev}{{\sf odev}}
\newcommand{\jer}{Je\v{r}\'abek}
\newcommand{\contheta}{\text{\textup\straighttheta}}

\newcommand{\fol}[1]{{\sf Th}_{#1}}

\newcommand{\iopen}{{\sf IOpen}}

\newcommand{\irred}{{\sf irred}}
\newcommand{\sg}{{\sf sg}}
\newcommand{\mbbs}[1]{\ensuremath{\mathsmaller{\mathbb{#1}}}}
\newcommand{\Y}{{\sf Y}}
\newcommand{\sm}{\hspace{0.025cm}}
\newcommand{\psv}{{\sf pvar}}
\newcommand{\psvar}{pseudo-variables}
\newcommand{\QS}{\mathbb{B}}
\newcommand{\m}{{\sf m}}
\newcommand{\rc}{{\sf rc}}
\newcommand{\zxg}{\ensuremath{\mathbb Z[\X]^{\mathsmaller{\geq 0}}}}
\newcommand{\zxxg}{\ensuremath{\mathbb Z[\,\X_q\mathop\mid q \in \mathbb{Q}]^{\mathsmaller{\geq 0}}}}
\newcommand{\T}{{\mathcal T}}
\newcommand{\skippy}{}
\newcommand{\shef}{\mathfrak{S}}
\newcommand{\fsvar}{Y}
\newcommand{\fsvard}{Y}

\title[Completions of Restricted Complexity]{Completions of Restricted Complexity I\\
Weak Arithmetical Theories}

\author{Ali Enayat}
\address{Department of Philosophy, Linguistics, and Theory of Science, University of Gothenburg, Sweden \texttt{(e-mail:~ali.enayat@gu.se)}}
\author{Mateusz \L{}e\l{}yk}
\address{Faculty of Philosophy, University of Warsaw, Poland (\texttt{e-mail:~mlelyk@uw.edu.pl})}
\author{Albert Visser}
\address{Department of Philosophy and Religious Studies, Utrecht University, The Netherlands (\texttt{e-mail: a.visser@uu.nl})}

\date{\today}

\begin{document}

\keywords{discretely ordered commutative ring, weak arithmetical theories, interpretations, restricted complexity, Open induction, Collection}

\subjclass[2010]{03F25,
03F30,
03F40,
}

\begin{abstract}
Given a first-order theory $T$ formulated in the usual language of first-order arithmetic, we say that $T$ is of 
\emph{restricted complexity} if there is some natural number $n$ and some set $\mc{A}$
of $\Sigma_n$-sentences such that $T$ can 
be axiomatized by $\mc{A}$. Motivated by the fact that no consistent arithmetical theory extending $\mathrm{I}\Delta _{0}+\mathsf{Exp}$ 
has a consistent completion that is of restricted complexity, we construct models of arithmetic whose complete theories are 
of restricted complexity. Our strongest result shows that there is a model of $\mathsf{IOpen + Coll}$ 
whose complete theory is of restricted complexity, where $\mathsf{Coll}$ is the full collection scheme.
\end{abstract}

\maketitle

\section{Introduction}\label{sec_introd}
Given a first order theory $T$ formulated in the usual language of first order arithmetic using the 
signature  $\{0, 1, +, \times, <\}$, we say that $T$ is of \emph{restricted complexity} if there is some natural number $n$ 
and some set $\mc{A}$ of $\Sigma_n$-sentences such that $T$ can be axiomatized by $\mc{A}$; here $\Sigma _{n}$ refers 
to the usual hierarchy of arithmetical formul{\ae}, in which formul{\ae} all of whose quantifiers are bounded are 
classified as $\Sigma _{0}=\Delta_{0}$, and the other classes are obtained by prefixing $\Delta_{0}$-formulae with appropriate strings 
of alternating blocks of quantifiers. Note that there are two independent ways to calibrate the complexity of an axiom set $\mc{A}$ 
of arithmetical sentences: the complexity of the set of G\"{o}del-numbers of $\mc{A}$ versus the $\Sigma_n$-complexity of each 
individual axiom in $\mc{A}$. It is the latter (and not the former) that is at stake here.

\skippy
The point of departure for the work presented here is the following recent incompleteness result. A proof is given in \cite [Theorem B]{Enayat-Visser2024}. 
As pointed out in \cite {Enayat-Visser2024}, an abstract version of the result was formulated by Emil \jer~ on MathOverflow in 2016.
We feel that such a fundamental insight should have been known much earlier, but we did not find it.

\begin{theorem} \label{motivatingTheorem}  Suppose $T$ is a consistent theory formulated in the language of arithmetic such that $\mathrm{I}\Delta _{0}+\mathsf{Exp}$ is provable in $T$. If $T$ is of restricted complexity, then $T$ is incomplete.
\end{theorem}

\noindent In the above, $\mathrm{I}\Delta _{0}$ is the result of restricting the induction scheme in the usual axiomatization of $\mathsf{PA}$ to $\Delta _{0}$-formulae, and $\mathsf{Exp}$ is the sentence asserting the totality of the exponential function $x\mapsto 2^{x}$. Theorem \ref{motivatingTheorem} naturally prompts the following question: 

\begin{jquestion} \label{motivatingQ}
Let $T_{0}$ be a consistent  theory formulated in the language of
arithmetic that extends $\mathsf{Q}$ (Robinson arithmetic) such that  $\mathrm{I}\Delta _{0}+\mathsf{Exp}$ is not provable in 
$T_{0}$\textsf{.} Is there a consistent completion of $T_{0}$ that is of restricted complexity?\footnote{It is known that Presburger arithmetic has a restricted axiomatization; this is a consequence of the elimination of quantifiers for Presburger arithmetic (see, for example, 4.5 on pp316, 317 and Exercise 9 on p328 in Craig Smory\'nski's book \cite{Smorynski_Logical_NT}). This also shows that  Question 1.2 has an easy answer if we remove the condition that $T_0$ includes $\mathsf{Q}$, since we can simply add an axiom to Presburger arithmetic stipulating that the multiplication of any two numbers equals 0.}
\end{jquestion}

For $T_{0}=\mathsf{I}\Delta _{0}$, the answer to Question \ref{motivatingQ}  appears to be beyond 
current methods, since the two different proofs of Theorem \ref{motivatingTheorem} as presented in \cite{Enayat-Visser2024} make it clear that the analogous result for $\mathsf{I}\Delta _{0}$ implies that $\mathrm{I}\Delta _{0}$ neither proves 
that there is a $\Sigma _{1}$-satisfaction
predicate (which is known to be a tall order, as demonstrated by Zofia Adamowicz, Leszek Ko\l odziejczyk, and Jeff
Paris in \cite {adam:truth12}), nor proves the MRDP theorem on the Diophantine representation of recursively enumerable sets (which is known to imply that NP = co-NP, as noted by Alex Wilkie \cite{Wilkie-Delta_0-def}). 

\skippy
In this paper, we show that Question \ref{motivatingQ} has a \emph{positive answer}
for certain well-known arithmetical theories $T_{0}$. Our strongest result provides a positive answer to Question \ref{motivatingQ} for $T_{0}=%
\mathsf{IOpen}+\mathsf{Coll}$, where $\mathsf{IOpen}$ is the result of
restricting the induction scheme in the usual axiomatization of $\mathsf{PA}$ to open
formul{\ae} (i.e., quantifier-free arithmetical formul{\ae}), and $\mathsf{Coll}$ is the full scheme of collection.\footnote{This is to be contrasted with the well-known fact that $\mathrm{I}\Delta
_{0}+\mathsf{Coll}$ axiomatizes $\mathsf{PA}$. By careful inspection of the proof
of this last insight in \cite[Lemma 7.4, p84]{kaye:mode91}, we see that this result
works already for the theory $\mathrm I{\sf E}_1$, which is, in a sense, 
\emph{the next theory} above
{\sf IOpen}. See Appendix~\ref{collectionsmurf} for a brief exposition.}

\skippy
Our work in this project was set in motion upon obtaining a positive answer to Question \ref{motivatingQ} for $T_{0}= \mathsf{PA}^{-}$ (\pam\ is the theory of the non-negative
part of an ordered commutative ring). More specifically we managed to demonstrate that the complete theory of the 
model \zxg\ of $\mathsf{PA}^{-}$ is of restricted complexity.
Here \zxg\ is the non-negative part of the ordered polynomial ring 
$\mathbb Z[\X]$, where $\X$ is treated as a positive infinite element. At the macro level, the proof of Theorem \ref{uitslagsmurf} employs the
framework of interpretability theory, while at the micro level, it relies on sequentiality of $\mathsf{PA}^{-}$ 
(demonstrated by Emil \jer~ \cite{jera:sequ12}) and
certain basic algebraic insights about \zxg. This led us to obtain an analogous positive result for 
$T_{0}=\mathsf{IOpen}$ with the same macro strategy as our first result. More specifically, we showed that the complete theory of 
the model of $\mathsf{IOpen}$ built by John Shepherdson using Puiseux series
in his ground-breaking paper \cite{shep:nons64} is of restricted complexity.

\skippy
The interpretability framework used to settle Question \ref{motivatingQ} for $T_{0}=\mathsf{PA}^{-}$ and $T_{0}=\mathsf{IOpen}$ employ definable injections of the universe into bounded range, and therefore a different strategy is needed for $T_{0}=\mathsf{PA}^-+\mathsf{Coll}$ and $T_{0}=\mathsf{IOpen}+\mathsf{Coll}$, since such definable injections are ruled out by the collection scheme. This led us to formulate a modified interpretability framework in which the notion of isomorphism is replaced with the notion of \emph{partial isomorphism} (in the back-and-forth style of Ehrenfeucht-Fra\"{\i}ss\'{e}
games). We used this new framework at the macro level, together with refined algebraic book-keeping arguments at the micro level to establish a positive answer to Question \ref{motivatingQ}  for $T_{0}=\mathsf{PA}^{-}+\mathsf{Coll}$ using the model \zxxg\ of $\mathsf{PA}^{-}+\mathsf{Coll}$ that was recently introduced in 
\cite{enlel:cat24} (here $\mathbb{Q}$ is the ordered set of rationals). This advance, in turn, inspired us to successfully apply the aforementioned modified interpretability framework based on partial isomorphisms, together with the necessary algebraic considerations, to build a model of $\mathsf{IOpen}+\mathsf{Coll}$ whose complete theory is of restricted complexity. Intuitively speaking, the relevant model is obtained by iterating the aforementioned Shepherdson construction $\mathbb{Q}$-times (as a direct limit). It is noteworthy that in all the models analyzed in this paper whose complete theory is of restricted complexity, the witnessing set $\mc{A}$ of $\Sigma_n$-sets consists of a single complex sentence plus an infinite set of $\Sigma_1$-sentences.\footnote{With more effort, these completions can be arranged to be also computably as simple as incompleteness allows, since each of the constructions in this paper can be carried out within a model of $\pa$ whose elementary diagram is $\Delta^0_2$ (instead of performing the construction within the standard model of $\pa$). This allows us to obtain a restricted completion that is additionally $\Delta^0_2$ as a set of natural numbers. Moreover, by taking advantage of Craig's trick, such a $\Delta^0_2$ completion can be axiomatized by a collection of axioms $\mathcal{A}$ such that $\mathcal{A}$ is $\Pi^0_1$ as a set of natural numbers.}
\skippy

There is a parallelism between the following three projects and our work here in constructing models of arithmetical theories whose theories are of restricted complexity.
\begin{enumerate}[$1$.]
    \item The construction of recursive (computable) nonstandard models of arithmetical theories, a topic boasting a large literature that originated with Shepherdson's pioneering paper \cite {shep:nons64} (which came on the heels of Tennenbaum's classical theorem on the nonexistence of computable models of Peano Arithmetic). Shahram Mohsenipour's \cite{MohseniIOpen} includes the latest known results and key references. 
    \item The determination of the arithmetical theories in which MRDP is provable. Here by MRDP we refer to the abstract version of the Matijasevi\v c-Robinson-Davis-Putnam theorem on the Diophantine representability of computably enumerable sets that states that every $\Sigma_1$-formula is provably equivalent to an existential formula, as in Richard Kaye's \cite{Kaye_Tennenbaum}. A recent contribution in this direction (and the one below) is by Yijia Chen, Moritz M{\"uller} and Keita Yokoyama \cite{Chen_Muller_Yokoyama}.
    \item Delimiting arithmetical theories that support a definable truth-predicate for $\Delta_0$-formulae. A key paper exploring this topic is \cite {adam:truth12}, which was mentioned earlier (right after Question \ref{motivatingQ}); a more recent contribution in this direction is by Konrad Zdanowski \cite{Zdanowski-Delta_0-truth}.
    \end{enumerate}
    In order to elaborate on the aforementioned parallelism, consider the set $\mathcal{T}_{\mathsf{Q}}$ of consistent theories formulated in the 
    language of arithmetic that extend $\mathsf{Q}$, and consider the following subclasses of theories $\mathcal{T}_{\mathsf{Q}}$ : 
\begin{itemize}
\item $\T_{\mathrm{MRDP}}$ is the set of $T\in \mathcal{T}_{\mathsf{Q}}$ such that MRDP is provable in $T$.
\item $\T_{\mathrm{truth}}$ is the set of $T\in \mathcal{T}_{\mathsf{Q}}$ such that $T$ supports a definable truth-predicate for $\Delta_0$-formulae.
\item $\T_{\mathrm{nrec}}$ is the set of $T\in \mathcal{T}_{\mathsf{Q}}$ such that $T$ does not have a nonstandard recursive (computable) model. 
\item $\T_{\mathrm{nrestr}}$ is the set of $T\in \mathcal{T}_{\mathsf{Q}}$ such that no consistent completion of $T$ is of restricted complexity.
\end{itemize}
As shown by Costas Dimitracopoulos and Haim Gaifman \cite{gaif:frag82} all arithmetical theories that extend $\mathrm{I}\Delta_0+\mathsf{Exp}$ are members of $\T_{\mathrm{truth}}\cap \T_{\mathrm{MRDP}}$. On the other hand, Kaye \cite [Theorem 3.3]{Kaye_Tennenbaum}, showed that $\T_{\mathrm{MRDP}}\subseteq\T_{\mathrm{nrec}}$. The two proofs of \cite [Theorem B]{Enayat-Visser2024} show that $ \T_{\mathrm{truth}} \subseteq \T_{\mathrm{nrestr}} $, and $ \T_{\mathrm{MRDP}} \subseteq T_{\mathrm{nrestr}} $. This is to be contrasted with the strongest result in this paper that shows that $\mathsf{IOpen +Coll}\notin \T_{\mathrm{nrestr}}$.
Constructing separating examples that distinguish the three classes $\T_{\mathrm{MRDP}}$, $\T_{\mathrm{truth}}$, and $\T_{\mathrm{nrec}}$ from each other appears to be a long-standing open question. We do not know whether the construction of examples that separate the class $\T_{\mathrm{nrestr}}$ from the aforementioned three is within reach. 

\skippy
A word about the organization of the paper is in order. Section 2 is devoted to relevant preliminaries concerning the arithmetical theories considered in this paper (Subsection 2.1), interpretability theory (Subsection 2.2), and salient algebraic facts (Subsection 2.3).  Section 3 is short but very important since it sets up the interpretability framework for all the constructions of the paper.  In Section 4, we show that the theory of \zxg\ is of restricted complexity, and in Section 5 we do the same for Shepherdson's model of $\mathsf{IOpen}$.  In Section 5, we show that the theory of the model \zxxg\
of $\mathsf{PA}^{-}+\mathsf{Coll}$ is of restricted complexity. At the technical level this section is considerably more involved than 
Sections 3 and 4 due to the complexity of the salient algebraic ideas and the nuances of the modified interpretability framework based on partial isomorphisms.  
Section 5 can be seen as a dress rehearsal for Section 6 in which we establish the culminating result of the paper by exhibiting a model of $\mathsf{IOpen}+\mathsf{Collection}$ whose theory is of restricted complexity.  

\section{Basics}

\subsection{Theories}
All our theories are theories of predicate logic of finite signature.
We write $\fol\Theta$ for the deductive closure of the axioms of predicate logic of signature $\Theta$.

For the purposes of this paper, an \emph{extensional} view of theories suffices. 
So, a theory will be given by a signature and a deductively closed set of sentences of that
signature.

\subsection{\pam\ and Extensions}\label{pamandextensions}
In this paper we are specifically interested in the theory \pam\ and weak theories extending it. \pam\ is the theory of the non-negative parts of discretely ordered
commutative rings. The signature of our theory is the signature with 0, 1, $+$, $\times$, $<$. We will follow the usual notational conventions for omitting brackets and in some cases
omitting $\times$ writing, e.g., $xy$ for $x\times y$.
Here are the axioms of the theory as given by Richard Kaye in \cite[Chapter 2.1]{kaye:mode91}.

\begin{enumerate}[{\sf pam}1.]
    \item 
    $\vdash (x+y)+z= x+(y+z)$
    \item $\vdash x+y = y+x$
    \item $\vdash (x\times y)\times z = x\times (y \times z)$
    \item $\vdash x\times y = y \times x$
    \item $\vdash x \times (y+z) = x\times y + x\times z$
    \item $\vdash x+0=x \wedge x\times 0 = 0$
    \item $\vdash x\times 1 = x$
    \item $\vdash (x< y \wedge y < z) \to x< z $
    \item $\vdash \neg\, x< x$
    \item $\vdash x<y \vee x=y \vee y<x$
    \item $\vdash x< y \to x+z< y+z$
    \item $\vdash (0<z \wedge x<y) \to x\times z < x\times y$
    \item \label{subtractionax} $\vdash x< y \to \exists z\; x+z=y$
    \item $\vdash 0< 1 \wedge (0<x \to (x=1 \vee 1<x))$
    \item $\vdash x=0 \vee 0<x$
\end{enumerate}

The theory \pam\ turns out to be a good base for weak arithmetical theories. We note that all axioms except {\sf pam}\ref{subtractionax} are universal; the system without this substraction axiom, sometimes jokingly called ${\sf PA}^{\mathsmaller{---}}$ (\pam\ -minus-minus), was studied by Emil  \jer\ in \cite{jera:sequ12}.
\jer\ verifies that we have a good sequence coding in this theory, i.o.w., that this
theory is sequential. For detailed information on sequential
theories, see \cite{viss:what13}. Sequentiality of a theory is inherited by its extensions. Thus, all theories studied in the present paper have the property.

We briefly present some of \jer's results for \pam, noting that they already hold in the weaker theory without subtraction.
First, \jer\ proves that $\tupel{x,y} := (x+y)^2+x$ is a pairing function in \pam.
We note that, in \pam, we have $x\leq \tupel{x,y}$ and $y \leq \tupel{x,y}$.
Then, \jer\ defines his version of the $\beta$-function:
\begin{itemize}
    \item {\small
   $\beta(x,i,w) :\iff \exists u,v,q\, (w=\tupel{u,v} \wedge u= q(1+(i+1)v)+x \wedge x \leq (1+(i+1)v))$. 
   }
\end{itemize}
Clearly, $u$, $v$ and $q$ can be bounded by $w$. So, this gives us a $\Delta_0$-definition.
Also, necessarily $x \leq w$.
\cite[Lemma 4(iii)]{jera:sequ12} tells us that $\beta$ is functional from $i,w$ to $x$.

We define:
\begin{itemize}
    \item 
    ${\sf seq}^\ast(s) :\iff \exists n,w \leq s\, (s = \tupel{n,w} \wedge \forall i \bles n \,\exists x\bleq w  \, \beta(x,i,w))$. 
\end{itemize}
In case $s$ is an $\ast$-sequence (i.e., $\mathsf{seq}^\ast(s)$ holds), 
Clearly, the length function has a $\Delta_0$-graph. We also write $\pi(s,i) =x$, whenever $s$ is an $\ast$-sequence $\tupel{n,w}$, $i<n$, and
$\beta(x,i,w)$. We note that the graph of $\pi$ is $\Delta_0$ and $\pi$ is total on $i<n$, whenever $s$ is an $\ast$-sequence.

A crucial result is \cite[Lemma 8]{jera:sequ12}. In our terminology, this says that, in \pam, for a certain definable inductive $J$, we have: (\dag)
for all $\ast$-sequences $s$ where the length of $s$ is $n$ in $J$ and for all $x$, there is an $\ast$-sequence $s'$ of length $n+1$, such that
$\pi(s,i)=\pi(s',i)$ for all $i<n$ and $\pi(s',n)=x$. 

In the nonstandard models of \pam we will consider, the standard cut $\omega$ is $\Delta_0$-definable, say by $\eta$. The property given by (\dag) is clearly inherited by definable inductive $J'$
that are shorter than $J$. Since $\omega$ is initial in any inductive $J'$, we will have (\dag) for $\omega$ as defined by $\eta$.
This inspires the following definition: ${\sf seq}(s)$ iff ${\sf length}(s) \in \eta$ and  ${\sf seq}^\ast(s)$.
Thus, {\sf seq} will deliver, inside $\mc M$, sequences in the usual sense.

The theory \pam\ and the theory of discretely ordered commutative rings have a close connection. We will give an explication of this connection in
Example~\ref{pamsyn}.

We will also consider the following extensions of \pam.
\begin{itemize}
    \item $\mathsf{IOpen}$ (\emph{Open Induction}). This is the extension of \pam obtained by augmenting \pam with the induction scheme for \emph{quantifier-free formulas} (parameters allowed).
    \item $\mathsf{PA}^{-} + \mathsf{Coll}$, where $\mathsf{Coll}$ is the \emph{collection scheme} consisting of formulae of the 
    following form in which parameters are allowed:\[\forall x \bles v\, \exists y\; \phi(x,y,\vv z) \to 
    \exists w \, \forall x\bles v\, \exists y\bles w\; \phi(x,y,\vv z).\]
    \item 
    $\mathsf{IOpen} + {\sf Coll}$.
\end{itemize}

\subsection{Basics of Interpretations}\label{interpretations}

In this section, we briefly sketch the basics of interpretations. 

\subsubsection{Translations and Interpretations}
We present the notion of \emph{$m$-dimensi\-onal interpretation without parameters}.
There are two extensions of this notion: we can  consider piecewise interpretations
and we can add parameters. We will not treat the piecewise case, but will briefly indicate how to
add parameters.

Consider two signatures $\Sigma$ and $\Theta$. An $m$-dimensional translation $\tau:\Sigma \to \Theta$
is a quadruple $\tupel{\Sigma,\delta,\mathcal F,\Theta}$, where $\delta(v_0,\ldots,v_{m-1})$ is a $\Theta$-formula and where,
for any $n$-ary predicate $P$ of $\Sigma$, $\mathcal F(P)$ is a $\Theta$-formula $\phi(\vv v_{0},\ldots, \vv v_{n-1})$ in the language
of signature $\Theta$, where
$\vv v_i = v_{i,0},\ldots, v_{i,(m-1)}$.
Both in the case of $\delta$ and $\phi$ all free variables are among the variables shown.
Moreover, if $i\neq j$ or $k \neq\ell$, then $v_{i,k}$ is syntactically different from $v_{j,\ell}$.

We demand that we have $\vdash {\mathcal F}(P)(\vv v_0,\ldots, \vv v_{n-1}) \to \bigwedge_{i<n} \delta(\vv v_i)$.
Here $\vdash$ is provability in predicate logic. This demand is inessential, but it is convenient to have.

We define $\psi^\tau$ as follows:
\begin{itemize}
\item
$(P(x_0,\ldots,x_{n-1}))^\tau :=  \mathcal F(P)(\vv x_0,\ldots, \vv x_{n-1})$.
\item
$(\cdot)^\tau$ commutes with the propositional connectives.
\item
$(\forall x\, \phi)^\tau := \forall \vv x \,( \delta(\vv x\,) \to \phi^\tau)$.
\item
$(\exists x\, \phi)^\tau := \exists \vv x \,( \delta(\vv x\,) \wedge \phi^\tau)$.
\end{itemize}
There are two worries about this definition. 
First, what variables $\vv x_i$ on the side of the translation $A^\tau$
correspond with $x_i$ in the original formula $A$? 
 The second worry is that  substitution of variables in $\delta$
and $\mathcal F(P)$ may cause variable-clashes.
These worries are never important in practice: we choose
 `suitable' sequences $\vv x$ to correspond to variables $x$, and we avoid clashes by re-naming bound variables, i.e., we apply $\alpha$-conversion. 
 However, if we want to give precise definitions of translations and, for example, of composition of translations, these
 problems come into play. The problems are clearly solvable in a systematic way, but this endeavor is beyond the scope of this paper.
 
 We allow the identity predicate to be translated to a formula that is not identity.

 A translation $\tau$ is \emph{identity preserving} if it translates identity to identity, i.e., if $(x=x)^\tau = (x=x)$. A translation is \emph{unrelativised}
 if its domain is the full domain of the theory, i.e., if $\delta_\tau(x) = (x=x)$.
 It is \emph{direct}, if it is one-dimensional and identity preserving and unrelativised.

There are several important operations on translations. 
\begin{itemize}
\item
${\sf id}_\Sigma$ is the identity translation. We take $\delta_{{\sf id}_\Sigma}(v) := v=v$ and $\mathcal F(P) := P(\vv v\,)$.
\item
We can compose translations. Suppose $\tau:\Sigma\to \Theta$ and $\nu:\Theta \to \Lambda$.
Then $\nu\circ \tau$ or $\tau\nu$ is a translation from $\Sigma$ to $\Lambda$. We define:
\begin{itemize}
\item
$\delta_{\tau\nu}(\vv v_0,\ldots,\vv v_{m_\tau-1}) := \bigwedge_{i< m_\tau}\delta_\nu(\vv v_i) \wedge (\delta_\tau(v_0,\ldots,v_{m_\tau-1}))^\nu$.
\item
$P_{\tau\nu}(\vv v_{0,0},\ldots,\vv v_{0,m_\tau-1},\ldots \vv v_{n-1,0},\ldots,\vv v_{n-1,m_\tau-1}) := \\
{\bigwedge_{i<n, j< m_\tau} \delta_\nu(\vv v_{i,j})}\, \wedge (P(v_0,\ldots, v_{n-1})^\tau)^\nu $.
\end{itemize}
\end{itemize}

A translation relates signatures; an interpretation relates theories.
An interpretation $\imath:U\to V$ is a triple $\tupel{U,\tau,V}$, where $U$ and $V$ are theories and $\tau:\Sigma_U \to \Sigma_V$.
We demand: for all theorems $\phi$ of $U$, we have $V\vdash \phi^\tau$. In a context where we have an interpretation $\imath$ based on $\tau$, we will often write
$\phi^\imath$ for $\phi^\tau$.

Here are some further definitions.

\begin{itemize}
\item
${\sf ID}_U:U\to U$ is the interpretation $\tupel{U,{\sf id}_{\Sigma_U},U}$.
\item
Suppose $\imath:U\to V$ and $\jmath:V\to W$. Then, $\imath\jmath := \jmath\circ \imath:U\to W$ is $\tupel{U,\tau_\jmath\circ \tau_\imath,W}$.
\end{itemize}

\noindent
A translation $\tau$ maps a model $\mathcal M$ to an internal model $\widetilde \tau (\mathcal M)$ provided that 
$\mathcal M$ satisfies the $\tau$-translations of the identity axioms (including $\exists x\, x=x$). Thus, an interpretation $\imath:U\to V$ gives us a
mapping $\widetilde \imath$ from ${\sf MOD}(V)$,  the class of models of $V$, to ${\sf MOD}(U)$, the class
of models of $U$. If we build a category of theories and interpretations, usually {\sf MOD}
with ${\sf MOD} (\imath):= \widetilde \imath$ will be a contravariant functor.

We can extend our notion of interpretation by adding a parameter-domain $\alpha(\vv v)$. The dimension of this domain is
unrelated to the dimension of the corresponding translation. The variables $\vv v$ of the parameter domain are admitted
as parameters in the translations for domain and predicate symbols (including identity). We demand that the interpreting theory
proves that the parameter domain is non-empty.
We say that $\imath = \tupel{U,\tau,\alpha,V}$ is an interpretation if, for all sentences $\phi$ of the signature of
$U$, if $U \vdash \phi$, then
$V \vdash \forall \vv v\, (\alpha(\vv v) \to \phi^{\tau,\vv v})$.

 We write $U^\tau$ for $\verz{\phi^\tau \mathop\mid U \vdash \phi}$.
 We note that if $A$ is finitely axiomatisable, then so is $A^\tau$. Suppose $\phi$ is a single axiom for
 $A$. We axiomatise $A^\tau$ by $\phi^\tau$ plus the $\tau$-translations of the identity axioms for
 the signature of $A$. These axioms include $\exists x\, x=x$ ensuring the non-emptiness of the domain.
 The addition of the translations of the identity axioms is necessitated, in many cases, by the
 translation being relativised and the fact that identity does not translate to identity.
 
 We note that, trivially, we have an interpretation $\imath$ of $U$ in $U^\tau$ based on translation $\tau$.
 In fact $U^\tau$ is the minimal theory $V$ such that there is an interpretation of $U$ in $V$ based on $\tau$.

In case our translation has parameters with parameter domain $\alpha$, the theory $U^\tau$ has to contain
$\exists \vv v\, \alpha(\vv v)$ and $\forall \vv v\, (\alpha(\vv v) \to \phi^{\tau,\vv v})$, whenever $U \vdash \phi$.

\subsubsection{Categories of Theories and Interpretations}

We introduce some relevant categories of theories and interpretations. The objects of these categories are theories \emph{considered as sets of theorems}.
We already introduced identity interpretations and composition of
interpretations, however these do not yet yield a category since e.g. ${\sf ID}_V \circ \imath = \imath$ fails (where $\imath:U \to V$). The solution is to
define an appropriate notion of sameness of interpretation. It turns out that there are quite different choices of these notions of sameness which lead
a.o. to quite different isomorphisms between theories. This  idea was studied in \cite{viss:cate06} for the one-dimensional case. In the present paper,
we will need two new such categories. The following categories are well-known.

\begin{itemize}
    \item 
${\sf INT}_0$ is the category where two interpretations $\imath,\jmath:U \to V$ are
the same if $V$ proves that the domains provided by both interpretations are the same and that the translations of the predicate symbols of $V$ are the same on these domains.
Isomorphism between theories in this category is \emph{definitional equivalence} or \emph{synonymy}.
\item ${\sf INT}_1$ is the category where two interpretations $\imath,\jmath:U \to V$ are
the same if there is a $V$-definable, $V$-verifiable isomorphism $F$ between $\imath$ and $\jmath$. 
Isomorphism in this category is \emph{bi-interpretability}.
\item 
${\sf INT}_2$ is the category where two interpretations $\imath,\jmath:U \to V$ are
the same if, for all $V$-models $\mc M$, the internal models $\mc M^\imath$ and
$\mc M^\jmath$ are isomorphic. Isomorphism in this category is \emph{iso-congruence}.
\item 
${\sf INT}_3$ is the category where two interpretations $\imath,\jmath:U \to V$ are
the same if, for all $U$-sentences $\phi$, we have $V \vdash \phi^\imath \iff \phi^\jmath$.
Isomorphism in this category is \emph{elementary congruence} or \emph{sentential congruence}.
\end{itemize}
Note that the above categories involve parameter-free interpretations, but they readily can be adapted to the parametric context. 

We note that, e.g., in ${\sf INT}_1$, 
there are various choices on how to set up the case of parameters. 
Moreover, in the parametric ${\sf INT}_1$-case,  the definable isomorphism
$F$ may contain extra parameters; in particular $F$ will be allowed to have parameters even if $\imath$ 
and $\jmath$ are parameter-free. This will be the case for our definable isomorphisms in 
Sections \ref{pamitself} and \ref{openitself}). 

Since `$F_x$ is an isomorphism', where $x$ is the parameter, can be presented as a single
formula, we can take the virtual class of $x$ such that $F_x$ is an isomorphism as parameter domain for
$F$. The only thing we need to verify is that there is an $x$ such that $F_x$ is an isomorphism.
So, in a sense, we can do without parameter domain.\footnote{This is analogous to the treatment of parameters
for interpretations of finitely axiomatized theories in \cite[p3]{myci:latt90}.} Our treatment 
in Sections \ref{pamitself} and \ref{openitself}) will follow this idea.

\begin{exa}\label{pamsyn}
     Consider \pam\ and the theory of discretely ordered commutative rings, denoted {\sf Docr}.
    We interpret \pam\ in {\sf Docr} simply by domain restriction, and we interpret {\sf Docr} in \pam, e.g., by
taking as domain the odd-or-evens and by
    letting $2n$ pose as $n$ and $2n+1$ as $-n-1$.\footnote{Alternatively, we can use the pair $\tupel{0,n}$ to represent the number $n$ and the pair $\tupel{1,n}$ to represent $-n-1$.}
    These two interpretations form a bi-interpretation.
    Using the main result of \cite{viss:biin25}, we can use this bi-interpretation to construct a
    definitional equivalence.
\end{exa}

We introduce two new categories. We treat the one-dimensional parameter-free case. 

The first, ${\sf INT}_3^{\sf fw}$ or
${\sf INT}_{2{\tt b}}$, is the category of
finitely witnessed elementary equivalence. Here two interpretations $\imath,\jmath:U\to V$
are the same if there is a finite sub-theory $V_0$ of $V$, such that
or all $U$-sentences $\phi$, we have $V_0 \vdash \phi^\imath \iff \phi^\jmath$.
It is easy to see that this gives indeed a category.
Suppose  $V_0$ witnesses the sameness of $\imath_0,\imath_1:U \to V$ and $W_0$ witnesses the
sameness of $\jmath_0,\jmath_1:V \to W$.
Then:
\begin{eqnarray*}
    W_0+ V_0^{\jmath_0} \vdash \phi^{\imath_0\jmath_0} & \iff & \phi^{\imath_1\jmath_0} \\
    & \iff & \phi^{\imath_1\jmath_1}
\end{eqnarray*}
Clearly, ${\sf INT}_3^{\sf fw}$ is between ${\sf INT}_3$ and ${\sf INT}_1$. It
is \emph{prima facie} incomparable with ${\sf INT}_2$.
Isomorphism in this category is \emph{finitely witnessed elementary congruence} or \emph{finitely witnessed sentential congruence}.

The second new category ${\sf INT}_{2{\tt a}}$ is the category where sameness
between $K,M: U \to V$ is witnessed by an Ehrenfeucht-Fra\"{\i}ss\'{e}  game. We give the definition
for the one-dimensional case without parameters. To get the definition off the
ground we need to assume that $V$ is sequential. So, the reasonable thing to
do is to demand that all theories in the category are sequential.
Thus, the correct name of our category really should be ${\sf INT}^{\sf seq}_{2{\tt a}}$.
The Ehrenfeucht-Fra\"{\i}ss\'{e}  game is given by a unary predicate $\eta$ such that the following properties
are $V$-verifiable.
\begin{itemize}
\item 
If $f\in \eta$ then $f$ represents a partial finite bijection between $\delta_\imath$ and $\delta_\jmath$.
Specifically, if $fxy$ and $fx'y'$, then $x,x' \in \delta_\imath$ and $y,y'\in \delta_\jmath$ and ($x=_\imath x'$ iff $y=_\jmath y')$.
    \item 
    Suppose $P$ is a $U$-predicate, $f\in \eta$, and $\vv x$ is in the domain of $f$.
    Then, $P_\imath(\vv x)$ iff $P_\jmath(f\vv x)$.
\item 
Suppose $f\in \eta$ and $x \in \delta_\imath$ and $x$ is not $\imath$-equal to an element in the domain of $f$. Then, there is a $g\in \eta$ with
\[\forall z \, \bigl(z\in {\sf dom}(g) \iff  (z\in {\sf dom}(f) \vee z=x)\bigr)\] and $\forall z\in {\sf dom}(f)\;f(z)=g(z)$.
(Note that two different $f$ can code the same input-output relation.)
    \item 
    Suppose $f\in \eta$ and $y \in \delta_\jmath$ and $x$ is not $\jmath$-equal to an element in the range of $f$. Then, there is a $g\in \eta$ with
\[\forall u \, \bigl(z\in {\sf range}(g) \iff  (u\in {\sf range}(f) \vee u=y)\bigr)\] and $\forall z\in {\sf dom}(f)\;f(z)=g(z)$.
\end{itemize}

By the usual argument, if $\imath$ and $\jmath$ are  equal, then they are {\sf fw}-elementary equivalent.
So, ${\sf INT}_{2{\tt a}}$ is between ${\sf INT}_{2{\tt b}}$ and ${\sf INT}_1$.

Isomorphism in ${\sf INT}_{2{\tt a}}$ will be: \emph{Ehrenfeucht-Fra\"{\i}ss\'{e} congruence}. 

\subsubsection{Some basic Results}

We start with two results about the category ${\sf INT}^{\sf fw}_3$. The first result will be useful in this paper.
The second will illustrate the power of sameness in the category.

\begin{theorem}\label{supersmurf}
    Let $\tau:\Xi\to \Theta$ and $\nu:\Theta \to \Xi$ be parameter-free translations. Let $\Phi$ be a $\Theta$-sentence such that
    $\Phi \vdash \fol{\Xi}^\tau$ and $\Phi \vdash \fol{\Theta}^{\nu\tau}$ and such that, for all $\Theta$ sentences $\phi$,
    we have $\Phi \vdash \phi^{\nu\tau}\iff \phi$. Consider any $\Theta$-theory $U$ that extends $\Phi$. Let $V$ be axiomatised by
    the $\Xi$-sentences $\psi$ such that $U\vdash \psi^\tau$. Then, there are interpretations $\imath:V \to U$ based on $\tau$ and
    $\jmath:U \to V$ based on $\nu$, that witness that $U$ is an ${\sf INT}^{\sf fw}_3$-retract of $V$. 
\end{theorem}

\begin{proof}
    Since $U \vdash \fol {\Xi}^\tau$, we indeed have an interpretation $\imath:V \to U$ based on $\tau$. We have $V \vdash \fol {\Theta}^\nu$. Suppose
    $U \vdash \phi$. Then, $U \vdash \phi^{\nu\tau}$ and, hence, $V \vdash \phi^\nu$. So, we indeed have an interpretation $\jmath:V \to U$ based on $\nu$.
    Finally, $\Phi$ witnesses that ${\sf id}_U$ and $\jmath\circ \imath$ are ${\sf INT}^{\sf fw}_3$-equal.
\end{proof}

\begin{theorem}
    Suppose $V$ is an ${\sf INT}^{\sf fw}_3$-retract of a finitely axiomatised theory $A$. Then, $V$ can be finitely axiomatised.
\end{theorem}

\begin{proof}
Suppose $\imath:V\to A$ and $\jmath: A\to V$. Let $\Phi$ witness the sameness of ${\sf ID}_V$ and $\jmath\circ \imath$.
Then, we have:\qedright
\begin{eqnarray*}
    V \vdash \phi & \To & A\vdash \phi^\imath \\
    & \To & A^\jmath\vdash  \phi^{\imath\jmath} \\
    & \To & A^\jmath+\Phi \vdash \phi \\
    & \To & V \vdash \phi
\end{eqnarray*}
\end{proof}

The following theorem tells us that if we have enough arithmetic, then interpretations can be reduced to a simpler form.
\begin{theorem}\label{smulpaapsmurf}
    Suppose $U$ proves that its domain is infinite and
    $\imath: U \to {\sf PA}$. Then there is a one-dimensional, identity preserving interpretation
    $\jmath : U \to {\sf PA}$ such that $\jmath$ is verifiably, definably isomorphic to $\imath$.
    Since in {\sf PA} we have a definable tuple of parameters, we can restrict the parameters of $\jmath$ to obtain
    a one-dimensional, identity preserving, parameter-free interpretation $\jmath^\ast : U \to {\sf PA}$.
\end{theorem}

\subsection{Bounded Quantifiers}
Here are some notational conventions and definitions. Let a signature $\Theta$ with designated binary relation symbol $\belowz$ be given.
\begin{itemize}
\item 
We write $\below{\vv x}{\vv y}$ for $\bigwedge_{i<n} \below{x_i}{y_i}$, where $n$ is the length of the sequences $\vv x$ and $\vv y$.
    \item 
    We  write $\forall \bbelow {\vv y}{\vv x}\, \phi$ for $\forall {\vv y}\, ( \below {\vv y}{\vv x} \to \phi)$ and 
    $\exists \bbelow {\vv y}{\vv x}\, \phi$ for $\exists {\vv y}\, ( \below {\vv y}{\vv x} \wedge \phi)$.
    \item 
    A uniform quantifier block is bounded if it can be written as $\forall \bbelow {\vv y}{\vv x}\, \phi$ or
    $\exists \bbelow {\vv y}{\vv x}\, \phi$.
    \item 
    A formula is $\Delta_0$, $\Sigma_0$ or $\Pi_0$ iff all its quantifiers are bounded.
    If $\phi$ is $\Sigma_n$, then the result of prefixing $\phi$ with a (possibly empty) block of bounded universal quantifiers
    is $\Pi_{n+1}$. If $\psi$ is $\Pi_n$, then the result of prefixing $\psi$ with a (possibly empty) block of bounded existential quantifiers
    is $\Sigma_{n+1}$. 
    \item 
    Let $U$ be a theory in our language. A formula is $\Delta_n[U]$ iff it is $U$-provably equivalent both to a $\Sigma_n$- and to a $\Pi_n$-formula.
    \item 
    Let $\Gamma$ be any formula-class of signature $\Theta$. The class
    $\Delta_0(\Gamma)$ is the result of closing off  $\Gamma$ under the propositional connectives and bounded quantification.
    \item 
    Let $\tau$ be a translation from $\Xi$ to $\Theta$. Let $\Gamma$ be any formula class. We say that $\tau$ is a $\Gamma$-translation iff
    $\alpha_\tau$, $\delta_\tau$ and each of the $P_\tau$ are in $\Gamma$.
\end{itemize}

\subsection{Dorroh-Rings}

Given a commutative ring $\mc R$, We say that $\mc R$ is a \emph{Dorroh-ring}, if there is an ideal $J$ of $\mc R$ such that
$\mc R/J \simeq \mathbb Z$. Such an ideal $J$ is said to \emph{witness} that $\mc R$ is a Dorroh ring. 
The polynomial ring $\mathbb Z [\X]$ is an example of a Dorroh ring, as witnessed by the ideal $J$ consisting of $p(\X)\in \mathbb Z [\X]$ such that $p(0)=0$. 

\begin{rem}\label{Dorroh_explanation}
    A \emph{non-unital ring}, aka \emph{rng}, is a ring that has no identity element. In Theorem 1 of \cite{dorr:conc32}, Joe Lee Dorroh gives a construction that adjoins an identity element to a non-unital ring; the proof makes it clear that the range of the Dorroh extension construction applied to all non-unital rings precisely consists, modulo isomorphism,
    of the Dorroh rings. This is the reason, we opted for the name \emph{Dorroh ring}.
\end{rem}

\begin{lem}\label{lem_char_dorroh}
Suppose $J$ is an ideal of the ring $\mc R$. Then, $J$ witnesses that  $\mc R$ is a Dorroh ring iff, for every element $a$ of $\mc R$, there is
    a unique pair $j\in J$ and $z\in \mathbb Z$ such that $a=j+z$.
\end{lem}

\begin{proof}
Suppose $\mc R$ is a Dorroh ring. By induction, we show that the witnessing isomorphism maps $\num n:=\overbrace{1+\ldots+1}^{n\mathrm{-times}}$ to $n$. It follows that
    every $J$- equivalence class contains either $\num n$ or $-\num n$. Moreover, $\num m\neq 0$, for $m\neq 0$ in
    $\mc R/J$, and, \emph{a fortiori}, in $\mc R$.  We identify the $\num n$ and $-\num n$ in $\mc R$ with $\mathbb Z$.
    It follows that every $a$ in $\mc R$ can be written as $j+z$, for some $j \in J$ and $z\in \mathbb Z$. 
    Suppose $j+z=j'+z'$, for $j,j'\in J$ and $z,z'\in \mathbb Z$. It follows that $j-j' = z'-z$. 
    So, $z=z'$ in $\mc R/J$, and hence in $\mc R$. So, also $j=j'$.

    Conversely,  suppose that, for every element of $a$ of $\mc R$, there is
    a unique pair $j\in J$ and $z\in \mathbb Z$ such that $a=j+z$.
    Ergo, there will be an element of $z$ in each equivalence class.
    There cannot be two integers in one equivalence class by uniqueness.
\end{proof}

\begin{lem}
    The witnessing ideal of a Dorroh ring is unique.
\end{lem}

\begin{proof}
    Let $\pi$, $\pi'$ be the projection functions, respectively, from $\mc R$ to $\mc R/J$ and
    from $\mc R$ to $\mc R/J'$ and let $f$ and $f'$ be the witnessing isomorphisms
    from $\mc R/J$ to $\mathbb Z$ and from $\mc R/J'$ to $\mathbb Z$.
    We have seen that $\pi\circ f$ and $\pi'\circ f'$ coincide on $\mathbb Z$
    as part of $\mc R$. It follows that any $a$ is uniquely of the form
    $j+z$ and $j'+z$. So, $j=j'$ and, thus, $J=J'$.
\end{proof}
In our context, an ordered commutative ring $\mathcal{R}$ comes equipped with a \emph{strict} total order $<$ that satisfies: 
(1) if $x<y$ then $x+z<y+z$; and (2) if $x<y$ and $z>0$, then $xz<yz$. Under this definition, an ordered ring is easily seen to be an integral domain.\footnote{Surprisingly, if (1) and (2) are reformulated by replacing $<$ with $\leq$, then the resulting definition no longer guarantees that the ring is an integral domain. One such example is the ring $\mathcal{R}=\mathbb{R}[\X]/J$, where $J$ is the ideal generated by $\X^{2}$, and the ordering is defined by $a+b\X  \leq c+d\X$ iff $a<c$, or $a=c$ and $b\leq d$. Note that $\X^2=0$, so $\mathcal{R}$ is not an integral domain.}
An \emph{ordered Dorroh ring} is an ordered commutative ring with the extra stipulation that, if $j>0$, then  $z<j$, for 
$z\in \mathbb Z$ and $j\in J$. We immediately see:

\begin{lem}
    An ordered Dorroh ring is discrete.
\end{lem}

We write $\mc R^{\mathsmaller{\geq 0}}$, for the non-negative part of an ordered ring $\mc R$.
The following theorem plays a key role in the proofs of our main results:
\begin{theorem}\label{Dorroh_fourquares}
Suppose $\mc R$ is an ordered Dorroh ring. Then,
    the natural numbers are $\Delta_0$-definable in $\mc R^{\mathsmaller{\geq 0}}$.
    It follows that the integers are definable in $\mc R$.
\end{theorem}

\begin{proof}
    We work in $\mc R^{\mathsmaller{\geq 0}}$. We define:
    \begin{itemize}
        \item 
   $x\in \omega_{\mathsmaller{\mc R}} $ iff
    $\forall y \bleq x\; \exists u_0 \bleq y\;\exists u_1\bleq y\;\exists u_2 \bleq y\,\exists u_3 \bleq y\;\; 
    y= u_0^2+u_1^2+u_2^2+u_3^2$.
     \end{itemize}
     Clearly, every natural number $n$ in $\mc R^{\mathsmaller{\geq 0}}$ does satisfy the condition, by
     Lagrange's Four Squares Theorem. 
     
     Consider any $a$ in $\mc R^{\mathsmaller{\geq 0}}$.  The element $a$ can be written as $j+z$, where $j\in J$ and $z\in \mathbb Z$.
     So, $a^2 = j^2+2zj + z^2$. Since,  $(j^2+2zj) \in J$, we find that $a^2$ is of the form $j'+z'$, where $z'\geq 0$.
     It follows that, no element $j^\ast+z^\ast$, for $z^\ast<0$, can be a sum of squares. Consider any element $b =j^\star +z^\star$,
     for $0\neq j^\star\in J$. Clearly, 
     $b-|z^\star|-1<b$ and $b-|z^\star| -1$ is not a sum of squares. So, $b$ is not in  $\omega_{\mathsmaller{\mc R}}$. 
\end{proof}

If we call a number \emph{four-squarable} if it is the sum of four squares, then $ \omega_{\mathsmaller{\mc R}} $ consists of the
\emph{hereditarily four-squarable} numbers of $\mc R$.\footnote{If $P$ is any property of numbers, $x$ is \emph{hereditarily $P$} if
$\forall y\bleq x\, P(y)$ and  $x$ is \emph{essentially $P$} if
$\forall y\mathop{\geq} x\, P(y)$.}

\begin{rem}
    Clearly, the only Dorroh ring of which the non-negative part satisfies Lagrange's Four Squares Theorem is
    $\mathbb Z$. 
\end{rem}

\begin{rem}\label{recursionsmurf}\
    The definability of $\omega$ in a Dorroh ring allows us to unravel recursive definitions. In this way we can obtain $\Sigma_1$ explicit definitions of functions defined by recursion such as the exponentiation \textup(with natural exponents\textup), generalized sum \textup(of true finite length\textup) et al.
\end{rem}

\subsection{Real Closure}\label{realsmurf} There are several equivalent formulations of the notion of a real closed field. For definiteness, we say that a real closed field is an ordered field $\mc K$ in which (1) every non-negative element has a square root, and every odd-degree polynomial in $\mc K[\X]$ has a root in $\mc K$. Given ordered fields $\mc K$ and $\overline{\mc K}$, 
we say that $\overline{\mc K}$ is \emph{a real closure of} $\mc  K$ if (1) $\overline{\mc K}$  extends $\mc K$ 
as an ordered field, (2) every element of $\overline{\mc K}$ is algebraic over $\mc K$, and 
(3) $\overline{\mc K}$ is real closed. By a classical theorem of Emil Artin and Otto Schreier, every ordered field has a real closure; moreover we have the following uniqueness feature of real closures:

\begin{theorem} \label{rc_is_unique} \cite[Theorem 2.9, Chapter XI]{lang:algebra2002} Any two real closures of an ordered field $\mc K$ are isomorphic over $\mc K$. 
\end{theorem}

Thus, every ordered field $\mc K$ has a unique real closure (up to isomorphism), which we will denote by ${\sf rc}(\mc K)$. 

Given an ordered commutative ring $\mc R$, we write ${\sf rc}(\mc R)$ for ${\sf rc}(\mathsf{fc}(\mc R))$, 
where $\mathsf{fc}(\mc R)$ is the field of quotients of $\mc R$.

Given certain ordered commutative rings $\mc R$ that are interpretable in the standard model $\mathbb{N}$ of arithmetic, we will need (in Sections \ref{openitself} and \ref{iterated_shepherdson}) to interpret $\mathsf{rc}(\mc R)$ in $\mathbb{N}$. Among the many ways of accomplishing this, we find James Madison's representation of ${\sf rc}(\mc R)$ in 
\cite{madison:comrcf70}) to be the most convenient for our purposes:
we start with all elements of the form $(f,n)$, where $f\in \mc R[\X]$ 
(here $f$ is represented as a finite sequence of elements of $\mc R$) and $n$ is a natural number, such that
\[{\sf rc}(\mc R)\models \underbrace{``\textnormal{There exists at least $n$ different roots of $f$}"}_{(*)}.\]
We observe that the above property can readily be verified already in $\mc R$: by quantifier elimination, for each $(f,n)$ 
as above there is a quantifier-free formula $\psi(\bar{x})$ such that 
${\sf rc}(\mc R)\models {(*)}\leftrightarrow \psi(\bar{b})$,
where $\bar{b}\in \mc R$ are all the coefficients of $f$. Since $\psi(\bar{b})$ is quantifier-free and $\bar{b}\in\mc  R$, the sentence $\psi(\bar{b})$ 
is absolute between ${\sf rc}(\mc R)$ and $\mc R$.  Let us call a pair $(f,n)$ satisfying $(*)$ an $\mc R$-notation. 
On the set of $\mc R$-notations, we consider the equivalence relation $\sim_{\mc R}$ defined as follows
\[(f,n) \sim_{\mc R} (g,k) \;:\Iff\; {\sf rc}(\mc R)\models ``\textnormal{The $n$-th root of $f$ is equal to the $k$-th root of $g$}".\]
Once again, whether $(f,n) \sim_{\mc R} (g,k)$ holds depends only on the atomic diagram of $\mc R$. 
Similarly, one can define addition and multiplication on the set of equivalence classes of $\sim_{\mc R}$. 
We define the inclusion map $\iota_{\mc R}: {\mc R}\rightarrow {\sf rc}(\mc R)$, called \textit{the standard embedding},  to be the mapping
\[\iota_{\mc R}(r):= [(\tupel{1,-r},1)]_{\sim_{\mc R}}.\]
We will follow the usual practice of identifying the elements from $\mc R$ 
with their images under $\iota_{\mc R}$.\footnote{This
is justified by an elementary fact of model theory that states that given an embedding 
$e:\mathcal{M}\hookrightarrow \mathcal{N}$ of a structure $\mathcal{M}$ as a submodel of a 
structure $\mathcal{N}$, there is a structure $\mathcal{N}^*$ such that $\mathcal{M}$ is a 
\emph{submodel} of $\mathcal{N}^*$, and there is an isomorphism 
$f:\mathcal{N}^{*} \rightarrow \mathcal{N}$ such that $f\circ i = e$, where 
$i:\mathcal{M} \hookrightarrow \mathcal{N}^{*}$ is the inclusion embedding} 
In the following, we will use the standard notation for some elements of the 
real closure of $\mc R$ and omit the brackets denoting the equivalence class. For example, instead of writing 
$[\tupel{1,0,-r},2]_{\sim_{\mc R}}$ we will be writing $\sqrt{r}$.

We observe that if $\mc R\subseteq \mc R'$ are both ordered commutative rings, then thanks to elimination of quantifiers in real closed fields, $\rc(\mc R)$ is an elementary submodel of $\rc(\mc R')$. In particular,
for each polynomial $f$ with coefficients in $R$, and for each $n\in\mathbb{N}$,
we find that $(f,n)$ is an $\mc R$-notation iff $(f,n)$ is an $\mc R'$-notation. 
Similarly, if $(f,m)$ and $(g,n)$ are two $\mc R$-notations, 
then $(f,m)\sim_{\mc R}(g,n)$ iff $(f,m)\sim_\mathsf{\mc R'}(g,n)$.

The above construction can be executed in the case of the ordered field of rational numbers $\mathbb{Q}$ and the presentation should make it clear that both the construction of $\mathbb{Q}$ and ${\sf rc}(\mathbb{Q})$ can be done in $\mathbb{N}$ (in fact the presentations of both fields are decidable).

\begin{rem}
The literature gives several ways to construct the real algebraic numbers in $\mathbb N$. What is more, they provide interpretations
of the theory {\sf RCF} of real closed fields in Robinson's arithmetic {\sf Q}. See the papers \cite{fern:groun22}, \cite{ferr:inte13} by Fernando Ferreira and Gilda Ferreira, the thesis \cite{dimi:firs19} (available online) by Anna Dmitrieva, and the following FOM posts by Harvey Friedman:\\
\begin{tabular}{l}
 {\tt https://cs.nyu.edu/pipermail/fom/1999-August/003345.html}\\
   {\tt https://cs.nyu.edu/pipermail/fom/1999-December/003523.html}
   \end{tabular}
\end{rem}

\section{The Framework}\label{smurfdebouwer}
In this section, we will build the foundation on which the results of our paper are built.
We will treat the case of one-dimensional parameter-free interpretations and morphisms between
interpretations with one parameter. 

\subsection{Interpretations that live below a Bound}
Consider a signature $\Theta$ with a designated binary predicate symbol $\belowz$ and an arbitrary signature $\Xi$.
Let $\tau$ be a parameter-free one-dimensional translation from $\Xi$ to $\Theta$ and
let $\theta(x)$ be a $\Theta$-formula with one free variable.\footnote{Even though in this paper $\belowz$ can be safely replaced with $\leq$ (or with $<$), we have opted for this more general approach to bounded quantifiers because the sequel of this paper uses $\in$ for bounding purposes.} 

We compute $\phi^{\tau,[x]}$ as follows: we follow the clauses of the inductive definition
of $\phi^\tau$ in the usual way, only we take:
\begin{itemize}
    \item $(\forall y\,\psi)^{\tau,[x]} := \forall \bbelow yx\,(\delta_\tau(y) \to \psi^{\tau,[x]})$,
    \item $(\exists y\,\psi)^{\tau,[x]} := \exists \bbelow yx\,(\delta_\tau(y) \wedge \psi^{\tau,[x]})$.
\end{itemize}

We define $\mf U(\tau,\theta)$ as the theory axiomatised by:
\begin{enumerate}[A.]
    \item $\fol\Xi^\tau$,
    \item $ \exists x\, x\in\theta \wedge
    \forall  y \bin \theta\, \forall  z\bin \delta_\tau \; \below zy$. 
\end{enumerate}

Clause (A) guarantees that $\tau$ carries an interpretation of Predicate Logic of signature $\Xi$ in $\mf U(\tau,\theta)$.
We have the following insight.

We say that a translation $\tau$ is a $\Gamma$-translation iff $\delta_\tau$ and the ${\sf P}_\tau$
are all in $\Gamma$.
\begin{lem}\label{bescheidensmurf}
Suppose $\tau$ is a $\Gamma$-translation and $\theta$ is a $\Gamma$-formula.
Let $\mf U:= \mf U(\tau,\theta)$.
Then $\phi^\tau$ is $\Delta_1[\mf U](\Gamma)$, for any $\Xi$-formula $\phi$.
\end{lem}

\begin{proof}
    We assume the conditions of the theorem. So,
    $\mf U + \theta(x) \vdash \phi^{\tau}\iff \phi^{\tau,[x]}$. Moreover,
    $\phi^{\tau,[x]}$ is $\Delta_0(\Gamma)$. 
So, we have, using that $\theta$ is $\mf T$-provably non-empty:
\begin{eqnarray*}
    \mf U \vdash  \exists x \bin \theta\;  \phi^{\tau,[x]} & \to & \phi^\tau \\
    & \to & \forall x\bin \theta\,  \phi^{\tau,[x]} \\
    & \to &  \exists  x\bin\theta \;  \phi^{\tau,[x]}
\end{eqnarray*}
It follows that $\phi^\tau$ is, modulo $\mf U$-provable equivalence, both in $\Sigma_1(\Gamma)$ and in $\Pi_1(\Gamma)$.
\end{proof}

\begin{rem}
    Let   $\contheta_\tau(x)$ be the formula $\forall y \in \delta_\tau\;  \below yx$.
Then, $\mf U(\tau,\theta)$ extends $\mf U(\tau,\contheta_\tau)$.
We work with a $\theta$ that possibly differs from $\contheta_\tau$, since such a $\theta$ is conceivably less complex than $\contheta_\tau$.
\end{rem}

\subsection{Bounds and Retracts}
 We have the following theorem.

\begin{theorem}\label{grotesmurf}
Let $\Gamma$ be a set of $\Theta$-formulas. Suppose that $\tau:\Xi \to \Theta$ is a $\Gamma$-translation and $\nu:\Theta\to \Xi$.
Suppose $\Psi$ is a $\Theta$-sentence such that $\Psi \vdash \fol{\Xi}^\tau$ and $\Psi \vdash \fol{\Theta}^{\nu\tau}$ and
$\Psi \vdash \phi \iff \phi^{\nu \tau}$, for all $\Theta$-sentences $\phi$. 
Let $\theta(x)$ be in $\Gamma$. Let $\Phi := \Psi \wedge \exists x\, x\in\theta \wedge
    \forall  y \bin \theta\, \forall  z\bin \delta_\tau \; \below zy$. 
Then, all $\Theta$-sentences are $\Delta_1[\Phi](\Gamma)$.
\end{theorem}
\begin{proof}
We assume the conditions of the theorem. By Lemma~\ref{bescheidensmurf}, the sentence
$\phi^{\nu\tau}$ is $\Delta_1[\Phi](\Gamma)$. By the $\Phi$-provable equivalence of
$\phi$ and $\phi^{\nu\tau}$, we find that $\phi$ is also in $\Delta_1[\Phi](\Gamma)$.
\end{proof}
We give the corollary that we will use in our arguments.

\begin{cor}\label{cor_mod_retract}
Let  $\mc{M}$ be a $\Theta$-structure, where $\Theta$ has a designated binary relation $\belowz$,
and let $\mc{N}$ be a structure in any signature.
Suppose we have one-dimensional parameter-free interpretations $\imath: \mc{N}\to \mc M$ and 
$\jmath: \mc M\to \mc{N}$ that witness that $\mc M$ is an ${\sf INT}_3^{\sf fw}$-retract of $\mc {N}$.
 Let $\tau$ and $\theta$ be in some prescribed formula class $\Gamma$.
Suppose that 
$\mc M \models  \exists x\, \theta(x) \wedge \forall x\bin \theta\, \forall y\bin \delta_{\imath}\; \below {x}{y}$.

Then, there are a $\Theta$-sentence $\Phi$, and a set 
of $\Delta_1[\Phi](\Gamma)$-formulas $\mc X$, such 
that $\Phi+\mc X$ axiomatises  $\mathrm{Th}(\mc{M})$.
\end{cor}
So,  specifically, if a theory $U$ has a model $\mc M$ satisfying the conditions of Corollary~\ref{cor_mod_retract}.
then, there is a completion of $U$ that is of the specified complexity.

\subsection{How to witness an  ${\sf INT}_3^{\sf fw}$-Retract}\label{Sec_witness}

In Sections \ref{pamitself} and \ref{openitself} we employ our first way to witness that a pair $\tau$, $\nu$ supports an 
${\sf INT}_3^{\sf fw}$-retract by constructing an ${\sf INT}_1$-retract.
In the paper, we will consider a definable isomorphism $\digamma_p(x,y)$ with parameter $p$
between $\tau \circ \nu$ and the identity translation on $\Theta$. The sentence $\Psi$ will
state that $\digamma_p$ is an isomorphism as desired, for some $p$.

In Sections \ref{smurferella} and \ref{iterated_shepherdson} we employ our second way, in which we construct an ${\sf INT}_{2{\tt a}}$-retract.
Here the sentence $\Psi$ consists of the atomic clauses and the
back-and-forth clauses of the Ehrenfeucht-Fra\"{\i}ss\'{e}  game between $\tau \circ \nu$ and the identity translation on $\Theta$.

\begin{question}\label{Q_about_retracts}
    Is there yet another way to construct a witnessing sentence for an ${\sf INT}_3^{\sf fw}$-Retract?
\end{question}

\section{A Restricted Completion of $\pam$}\label{pamitself}
Our basic argument for the simplest case works as follows. We produce a model $\mathbb M$ of \pam, in which
we can define the standard cut $\omega$ and specify a non-empty set of elements above $\omega$. 
We build an internal copy of $\mathbb M$ inside $\omega$. Using sequentiality, we construct
a definable isomorphism between the internal copy and the original model. This makes the framework of
Section~\ref{smurfdebouwer} applicable.

\subsection{The Second Standard Model}\label{secondstandard}
We use the terminology \emph{the second standard model \textup(of $\pam$\textup)} 
to refer to the model $\mathbb M:= \mathbb Z[\X]^{\mathsmaller{\geq 0}}$, i.e., 
the non-negative part of $\mathbb Z[\X]$ considered as an ordered ring. 
Specifically, $\mathbb M$ consists of 
the polynomials  $p(\X)=a_n \X^n+ a_{n-1} \X^{n-1} + \cdots + a_0$, with integer 
coefficients, where $a_n\geq 0$ and ($a_n\neq 0$, if $n \neq  0$), with the obvious operations. 
The ordering is determined by stipulating that $\X > \omega$. More formally, we first deem $p(\X)>0$ if $a_n>0$, and then we define the full ordering  by:  $p_1(\X)>p_2(\X)$ iff $p_1(\X)-p_2(\X)>0.$
We identify the standard numbers $\omega$ in the usual way
with constant polynomials. 

\subsection{$\Delta_0$-definitions of Salient Classes}
Since $\mathbb M$ is the non-negative part of the Dorroh ring $\mathbb Z[\X]$, we can give a $\Delta_0$-definition of the natural numbers $\omega_{\mbbs M}$
using the four squares theorem. In other words, these will be the hereditarily four-squarable numbers.
This definition is just one of the many ways to define the naturals by a $\Delta_0$-formula. Here is an alternative way of doing it.
We define the natural numbers as the hereditarily odd-or-even numbers. Formally:
\begin{itemize}
\item
The set of numbers that are odd or even is defined as follows:\\
  $x\in \odev := \exists y \bleq x\, (x =  2 y \vee x =  2 y+1) $.
 \item
 We now can define $\omega_{\mbbs{M}}$ by:
 $x\in \omega_{\mbbs{M}} := \forall y \bleq x\; x\in \odev$.
 The elements of $\omega_{\mbbs{M}}$ are the \emph{hereditarily odd-or-even numbers}.
\end{itemize}

We define sequences in our model $\mathbb M$ using the $\Delta_0$-defined $\omega_{\mbbs{M}}$
as discussed in Section~\ref{pamandextensions}.

\subsection{Interpreting $\mathbb N$ in $\mathbb M$.}
We simply do this by relativising the quantifiers to $\omega_{\mbbs M}$. Let us call this interpretation {\sf N}.

\subsection{Interpreting $\mathbb M$ in $\mathbb N$.}
We can build a copy of $\mathbb M$ inside $\mathbb N$ in any number of ways, seeing that
we have the full power of $\mathbb N$ at our disposal. Our development will not depend on any specific
feature of how we do it. 

Here is one sketch of one way to do it. We first construct an internal model of $\mathbb Z$
inside $\mathbb N$. This can be done in many ways. For specificity, we use the following one.
We employ $\tupel{0,n}$ for $+n$ and $\tupel{1,m+1}$ for $-(m+1)$. We note that this interpretation is identity preserving.

We describe a polynomial $a_n{\sf X}^n+\dots+ a_0$ as a sequence $(z_n,\dots, z_0)$, where
the $z_i$ are the representatives in our copy of $\mathbb Z$ of the integers corresponding to the $a_i$.
We demand $a_n \geq 0$, and $a_n>0$, if $n\neq 0$.
We can use the sequences given by the $\beta$-function or, alternatively, the
the sequences can defined by iterated pairing or using any other trick to construct sequences. 
However, we will employ the same notations for the projection- and length-functions
as we used in the $\beta$-case.

Defining addition and multiplication for these sequences corresponding to addition and multiplication of polynomials is clearly within the resources
of $\mathbb N$. This gives us an $\mathbb N$-internal model $\mathbb M^\ast$ that is
isomorphic to $\mathbb M$.  
Let the interpretation of $\mathbb M^\ast$ in $\mathbb N$ be {\sf M}. 
By Theorem~\ref{smulpaapsmurf}, we can make our interpretation identity-preserving, but nothing in what we are doing depends on that.

\subsection{The Application Function}\label{toepassingsgerichtesmurf}
We define the function {\sf app}. 
This function takes as input a representation $a$ of a polynomial and an element $x$ and gives as output the
value of the polynomial applied to $x$.
\begin{itemize}
\item 
$x\oplus \tupel{i,n} := x+n$, if $i=0$ and $x\oplus \tupel{i,n} := x-n$, if $i=1$ and $0<n\leq x$ and
$x \oplus \tupel{i,n} = 0$, otherwise. We note that minus is definable here by the subtraction axiom of
\pam.
\item
    ${\sf app}(a,x,b)$ iff, there is a sequence $s$ of the same length as $a$, such that $\pi(s,0) = \pi(a,0)$ and, 
    for all $i$ such that $i+1 <{\sf length}(a)$, 
    $\pi(s,i+1)= (\pi(s,i)\cdot x) \oplus \pi(a,i+1)$, and
    $\pi(s,{\sf length}(a)-1)=b$.
\end{itemize}
It is easy to see that {\sf app} indeed defines application as long as $x$ is so large that
the third clause of the definition of $\oplus$ is never active in the computation of the value.
This will certainly be the case if $x\not \in \omega_{\mbbs{M}}$ and this is the case we are interested in.

We note that ${\sf app}(\overbrace{\tupel{1_{\mbbs{Z}},0_{\mbbs{Z}},\dots,0_{\mbbs{Z}}}}^{\text{length } n+1},x) = x^n$. Since the function
$n \mapsto \overbrace{\tupel{1_{\mbbs{Z}},0_{\mbbs{Z}},\dots,0_{\mbbs{Z}}}}^{\text{length } n+1}$ is clearly arithmetically definable,
the function $n \mapsto x^n$ is $\mathbb M$-definable.

\subsection{Completing \pam}
We  apply Corollary~\ref{cor_mod_retract} with:
\begin{itemize}
    \item 
$\mathbb M$ in the role of $\mc M$, 
\item 
$\leq$ in the role of $\belowz$ ,\footnote{Here we can equally well use $<$ instead of $\leq$, since we are interested in bounding the standard numbers by a non-standard one.}  
\item 
$\mathbb N$ in the role of
$\mc N$, 
\item 
{\sf N} in the role of $\imath$,
\item 
{\sf M} in the role of $\jmath$. 
\item
$\omega_{\mbbs M}^{\sf c}$, the complement of $\omega_{\mbbs M}$ in $\mathbb M$, in the role of $\theta$.
\item 
Let $\mf F_{x}(a) := {\sf app}(a,x)$. Then, $\mf F_{\X}(a)$  is an isomorphism between
$\mathbb M^\ast$ and $\mathbb M$. The statement that $\mf F_x$ is an isomorphism
for some $x$ delivers our desired witness of an ${\sf INT}^{\sf fw}_3$-retract.
\end{itemize}

It is easy to see that our choices realise the conditions of the Corollary.
We note that $\tau_{\sf N}$ and $\omega_{\mbbs M}^{\sf c}$ are $\Delta_0$.
So it follows that we have a sentence $\Phi$, and a set $\mc X$
of $\Delta_1$-formulas $X$  (over $\Phi$), such 
that $\Phi + \mc X$ axiomatises  $\mathsf{Th}(\mathbb{M})$.

We conclude:

\begin{theorem}\label{uitslagsmurf}
    The theory \pam\ has a completion that is axiomatised by a sentence $\Phi$ plus a set of $\Sigma_1$-sentences. 
\end{theorem}

A nice consequence of our development is:
\begin{theorem}
    In $\mathbb Z[\X]$, we can define composition $P(\X) \circ Q(\X) := P(Q(\X))$ in parameter $\X$.
\end{theorem}

\begin{proof}
    We consider a one-dimensional identity-preserving version of {\sf M}. Then $\mf F_{\X}$ is a bijection.
    We take $P(\X)\circ Q(\X) := {\sf app}(\mf F^{-1}_{\X}(P(\X)),Q(\X))$.
\end{proof}

\subsection{Free Riders}

We stated Theorem~\ref{uitslagsmurf} for ${\sf PA}^-$, but ${\sf Th}(\mathbb M)$
satisfies many arithmetically true extensions of ${\sf PA}^-$, for example, any
arithmetically true sentence relativised to {\sf N}. These are the free riders
of the argument.  

However, many of such examples are somewhat contrived. Here is an example that is not
contrived. Since $\mathbb Z$ is a unique factorisation domain, so is $\mathbb Z[\X]$.
Clearly, the units in this case are $+1$ and $-1$. It follows that the elements of
$\mathbb M$ have unique factorisation with unit $1$. Using the existence of sequences,
as given by ${\sf seq}^\ast$,
we can define in $\mathbb M$ a partial function $\Pi$ that computes the product of the
elements of an $\ast$-sequence. 
We define:
\begin{itemize}
\item 
${\sf pseq}^\ast$ iff 
$s$ is a $\ast$-sequence and the elements
of $s$ are all prime and, for all $i<j<{\sf length}(s)$, we have $\pi(s,i)\leq \pi(s,j)$. 
\item 
$s\approx s' := {\sf length}(s) = {\sf length}(s') \wedge \forall i\bles {\sf length} (s) \; \pi(s,i)=\pi'(s',i)$.
    \item 
    ${\sf UF} := \forall x\, \exists s\bin {\sf pseq}^\ast \,\forall s'\bin {\sf pseq}^\ast \; (\Pi(s')=x \iff s \approx s')$.
\end{itemize}

Then we have: 
\begin{theorem}\label{uitslag2smurf}
    The theory $\pam+{\sf UF}$ has a completion that is axiomatised by a sentence $\Phi$ plus a set of $\Sigma_1$-sentences. 
    Specifically, this completion can be taken to be ${\sf Th}(\mathbb M)$.
\end{theorem}

Of course, there are all kinds of consequences of unique factorisation that are easier to formulate. An example is the existence
of greatest common divisors.

\begin{rem}    
We note that we do \emph{not} have the Euclidean Division Axiom internally in our model: think of dividing $\X$ by 2. We can modify our model
by allowing the coefficients $a_i$ for $i\neq 0$ to be in $\mathbb Q$. Thus, we obtain the model
$(\X\cdot \mathbb Q[\X] +\mathbb Z)^{\mathsmaller{\geq 0}}$. We can repeat our development almost verbatim
for this model. It satisfies the Euclidean Division Axiom internally.
We note that the associated discretely ordered commutative ring is  \emph{not} a 
Euclidean Domain. If it were, it would also be a Factorisation Domain, but clearly $\X$ does not have a factorisation
in irreducible elements.
However, the ring is a B\'ezout domain, so we still have greatest common divisors.
See:
{\tt https://en.wikipedia.org/wiki/B\'ezout\_domain}
\end{rem}

\subsection{Definitional Equivalence}

We have found an interpretation {\sf N} of (an isomorphic copy of) $\mathbb N$ in $\mathbb M$ and an interpretation {\sf M} of
(an isomorphic copy of) $\mathbb M$ in $\mathbb N$. These two interpretations form a bi-interpretation as witnessed by $\mf F_{\sf X}$
and a suitable parameter-free $\mf G$. Clearly, {\sf N} is one-dimensional and identity-preserving. 
Theorem~\ref{smulpaapsmurf} tells us that we can 
always assume that {\sf M} is one-dimensional and identity-preserving. If we expand the signature of the 
language of $\mathbb M$ with
a constant for \X, the witnessing isomorphism $\mf F_\X$ becomes parameter-free. 
We may now apply the main theorem of \cite{viss:biin25}, and conclude that $\mathbb N$ and 
the expansion $(\mathbb M,\X)$ of $\mathbb{M}$ are definitionally equivalent. 

\section{A restricted Completion of \iopen}\label{openitself}
Recall that \iopen\ is the theory in the arithmetical signature that is the result of strengthening $\pam$ with all 
instants of the induction scheme \emph{for quantifier-free formulae.} We redirect the reader to \cite[Chapter I]{haje:meta91} for information on how much arithmetic can be developed in \iopen.

\subsection{Shepherdson's Model}
In this subsection we analyze the recursive nonstandard model of $\iopen$ constructed by John Shepherdson in his pioneering paper \cite{shep:nons64}; we denote this model with $\mathbb{S}$. The elements of $\mathbb{S}$ are \emph{Puiseux polynomials} over the field $\rc(\mathbb{Q})$ of real algebraic numbers. More specifically, the elements of $\mathbb{S}$ are formal expressions of the form\footnote{Shepherdson interchanges the positive and negative
powers of Puiseux series in which $\X$ is deemed to be an infinitesimal (and thus $\X^{-1}$ becomes infinite). We will follow his lead here.}:
    \[p(\X)=a_n\X^{n/q}+a_{n-1}\X^{(n-1)/q} + \ldots + a_1 \X^{1/q} + a_0,\]
where each $a_i \in \rc(\mathbb{Q})$, $a_0\in \mathbb{Z}$, $n\in \mathbb{N}$, and $q\in \mathbb{N}\setminus\{0\}$; furthermore $a_n>0$ if $n\neq 0$; and $a_n\geq 0$ if $n=0$. As usual, we say that $a_n$ is the leading coefficient of $p$ as above. 
Addition and multiplication in $\mathbb{S}$ are given by the usual addition and multiplication operations on polynomials, and the ordering is given in such a way as to make $\X$ infinite. More specifically, for $p, q\in\mathbb{S}$, $p> q$ iff the leading coefficient of $p-q$ is positive.

The above presentation of elements of $\mathbb{S}$ is non-unique, because some of the $a_i$'s might equal $0$. By omitting the $0$ terms, 
we obtain the notion of a \textit{standard form}, which is an expression of the form
\[a_{n}\X^{q_n} + a_{n-1}\X^{q_{n-1}} + \ldots + a_1\X^{q_1} + a_0,\]
where

\begin{enumerate}[$1$.]
    \item each $a_i\in \rc(\mathbb{Q})$, and $a_i\neq 0$ if $i\neq 0$,
    \item $a_n \geq 0$,  
    \item $a_0\in\mathbb{Z}$, 
    \item each $q_i\in\mathbb{Q}^{\mathsmaller{> 0}}$, and  $q_i>q_j$ for $i>j$.    
\end{enumerate}

\subsection{The Natural Numbers in $\mathbb S$}
Since $\mathbb S$ is the non-negative part of a Dorroh ring (as witnessed by the ideal consisting of $p(\X)\in\mathbb{S}$ 
such that $p(0)=0$), we can define the natural numbers using the four squares theorem.
There are, of course, other possibilities that we briefly discus below.

Since the notions ``prime" and ``irreducible" coincide in $\mathbb{Z}$, the following proposition implies that the irreducible elements of $\mathbb{S}$ are precisely the prime numbers in $\mathbb{N}$: 
\begin{prop}\label{nijveresmurf}
    If $a$ is irreducible in $\mathbb{S}$, then $a\in \mathbb{N}$.
\end{prop}

This result is well-known from the literature. For completeness, we provide a proof in Appendix~\ref{robuustesmurf}.
We can now alternatively define:
\begin{itemize}
    \item
$\omega_{\mbbs{S}}(x):= x=0 \vee \exists y \bleq 2x\,
    \bigl(x\leq y \wedge \irred(y)\bigr)$.
    \end{itemize}
    
    \noindent
    The above definition works by Proposition~\ref{nijveresmurf} in combination with \emph{Bertrand's postulate}: for
    every natural number $n>0$ there is a prime number between $n$ and $2n$. 
    We note that Bertrand's postulate is a classical theorem of Number Theory (which was established by Chebyshev).

\subsection{Representing the real algebraic Numbers in $\mathbb S$}\label{sect_repr_real_alg}
Locally, we will use $\alpha,\beta,\dots$ elements of $\rc(\mathbb{Q})$, i.e., for real algebraic numbers. 

We can take over the definition of $x^n$ (where $x$ ranges over $\mathbb S$), for $n\in \omega$ from Section~\ref{toepassingsgerichtesmurf}.
It follows that we can also define $x^{n/m}$ for $m\in \omega$ as the maximal solution $y$, if any, of $x^n = y^m$.
We note that this function does not need to be total.

We remind the reader that we constructed codes of real algebraic numbers in $\mathbb{N}$ in Section \ref{realsmurf}. As we can define $\mathbb{N}$ in $\mathbb{S}$, these codes are available to us in $\mathbb{S}$ as well. Now we would like to map the codes $a$ of the real algebraic numbers to the real algebraic number as they occur
in $\mathbb S$. However, a real algebraic number $\alpha$ that is not a natural number is not a first-class citizen in $\mathbb S$; such a number $\alpha$  only occurs
in contexts like $\alpha\sm\X^{q}$, where $q\in \mathbb{Q}^{\mathsmaller{> 0}}$. To solve this problem, we map an $a\in\omega_{\mathsmaller{\mathbb S}}$ that represents $\alpha\in\rc(\mathbb{Q})$ (as in our adopted coding of real closures discussed in Subsection \ref{realsmurf}) in combination
with $q,r\in \omega_{\mathsmaller {\mathbb S}}$ to first-class citizen $\alpha\sm\X^{q/r}$.
Let $a= \tupel{p,k}$ and $p=\tupel{z_n,\ldots,z_0}$. 
We define:
\[Q_{x,p}(\Y) := z_n \Y^n + z_{n-1}\Y^{n-1}x + \ldots + z_1\Y x^{n-1} + z_0x^{n}.\]
We can now define an application function ${\sf app}^\ast$ in analogy of {\sf app} in Section~\ref{toepassingsgerichtesmurf},
where ${\sf app}_{x}^\ast(p,y) = Q_{x,p}(y)$.

Clearly, $\alpha$ is a solution of 
$P_p(\Y)=0$ in \Y\ iff  $\alpha\sm\X^{q}$ is a solution of $Q_{\X^{q},p}(\Y)=0$ in \Y. Moreover, the order of the corresponding solutions is preserved. 

Let $a := \tupel{p,k}$. We define the partial function ${\sf ev}_z$ that is chosen in such a way that
${\sf ev}_{\X^q}(a)= \alpha\X^q$, where $q\geq 0$.
\begin{itemize}
    \item 
    ${\sf ev}_{z}(a) = w$ iff there is an $s\in {\sf seq}$ such that:
    \begin{itemize}
    \item
    ${\sf length}(s)=k$,
        \item 
    $\pi(s,k-1) = w$,
    \item 
    for all $i<k$, we have ${\sf app}^\ast_{z}(p,\pi(s,i)) =0 $,
    \item
    for all $i<j<k$, we have $\pi(s,i) < \pi(s,j)$,
    \item 
    if $y \leq w$ and 
    ${\sf app}^\ast_{z}(p,y) =0$, then, for some  $j<k$, we have $\pi(s,j)= y$.
    \end{itemize}
\end{itemize}
We note that ${\sf ev}$ is a partial function. If we apply it to the correct inputs where $a$ is
a representation of an algebraic number and $z=\X^q$, for $q\in \mathbb Q^{\mathsmaller{\geq 0}}$, then it will
deliver the desired value.

\subsection{Interpreting $\mathbb N$ in $\mathbb S$}
As in the case of $\pam$, the interpretation of $\mathbb{N}$ in $\mathbb{S}$, denoted ${\sf N_{\mbbs{S}}}$, 
is given by restricting the domain of $\mathbb{S}$ to $\omega_{\mbbs{S}}$. 

\subsection{Interpreting $\mathbb S$ in $\mathbb N$}
We construct an isomorphic copy $\mathbb S^\ast$ of $\mathbb S$ in $\mathbb N$.

An element of $\mathbb{S}$ will be represented via its standard form, that is as a (code of a) sequence of pairs
$\langle \tupel{a_0,0},\tupel{a_1,q_1}\ldots, \tupel{a_n,q_n}\rangle$, where the $a_i$ are representatives of real algebraic numbers, $q_i$'s are rational numbers both satisfying the criteria for the standard form. We shall often implicitly transform this representation to a pair of sequences, denoted $\tupel{\bar{a},\bar{q}}$.
Identity, the ordering, zero, one, addition and multiplication on these objects can be defined as a matter of course.  We call our interpretation {\sf S}.

\subsection{The Isomorphism}\label{lachendesmurf}
Finally, let us give the formula $\digamma_z(x,y)$, which will define the isomorphism for $z:\X$, from the
$\mathbb N$-internal model representing $\mathbb S$ to $\mathbb S$.
We define:

\skippy
\begin{itemize}
    \item $\digamma_{z}(x,y)$ iff $x = \tupel{\bar{a},\bar{q}}$ is an $\mathbb N$-representation of an element of $\mathbb S$ and\\
    $y = \sum_{j<{\sf length}(x)}{\sf ev}_{z^{q_j}}(a_j)$.
\end{itemize}

\noindent
Here the sum is defined in the obvious way by recursion, which we do have since we have standardly finite sequences.

\skippy
Note that \noindent $\digamma_{z}(x,y)$ is equivalent to a $\Sigma_1$-formula.

\begin{theorem}\label{eigenwijzesmurf}
    There is a completion of \iopen\ that is axiomatisable by a single sentence together with an infinite set of $\Sigma_1$-sentences.
\end{theorem}
\begin{proof}

This follows immediately from Corollary \ref{cor_mod_retract}  with:
\begin{itemize}
    \item 
$\mathbb S$ in the role of $\mc M$, 
\item 
$\leq$ in the role of $\belowz$, 
\item 
$\mathbb N$ in the role of
$\mc N$, 
\item 
$\mathsf{N}_{\mathbb{S}}$ in the role of $\imath$,
\item 
{\sf S} in the role of $\jmath$. 
\item
$\omega_{\mbbs S}^{\sf c}$ in the role of $\theta$.
\item 
$\digamma_{\X}(x,y)$ is an $(\mathbb{S}, \X)$-definable isomorphism between
$\mathbb S^\ast$ and $\mathbb S$.
So, the statement that, for some $z$, the function $\digamma_z$ is an isomorphism
 delivers our desired witness of an ${\sf INT}^{\sf fw}_3$-retract.
\end{itemize}
\end{proof}
Shepherdson's construction of the model $\mathbb{S}$ led to a substantial body of work on the construction of other recursive models of $\mathsf{IOpen}$ that satisfy desirable algebraic properties, such as normality (which implies that $\sqrt{2}$ is irrational), the Bezout property, the GCD property, and having unboundedly many primes, all of which fail in $\mathbb{S}$. Mohsenipour's paper \cite{MohseniIOpen} contains the latest results in this direction\footnote{It is easy to see that the ring associated with $\mathbb{S}$ is not a factorization domain since $\X$ can be written as the product of $\frac{1}{2}\X$ and $2$ in $\mathbb{S}$, neither of which is a unit (invertible). So here there are no significant free riders. } This prompts the following question, which we suspect has a positive answer.
\begin{question} \label{IOpen_question} 
Is there a model of $\mathsf{IOpen}$ whose complete theory has a restricted axiomatization, and which satisfies normality and other desirable algebraic properties?
\end{question}

\section{A restricted Completion of \pam\ plus full Collection}\label{smurferella}
In this section, we present a model of $\pam+ {\sf Coll}$ whose complete theory is axiomatized by a collection of axioms consisting of a single sentence together with an infinite number of $\Sigma_1$-sentences. Note that ${\sf Coll}$ fails in the models $\mathbb{M}$ and $\mathbb{S}$ (using sequentiality, the definability of the standard cut in both models, and the fact that in both models finite powers of $\X$ are cofinal).

\subsection{Basics of the Model $\mathbb A$} 
Let $\mathbb{Q}$ be the ordered set of rationals. In this section we develop the basic algebraic properties of the discrete ordered semi-ring $\mathbb A$, defined as follows: $$\mathbb A :={\hat{\mathbb{A}}}^{\mathsmaller{\geq 0}},~\textrm{where~}\hat{\mathbb{A}}=
\mathbb Z[\X_q\mathop\mid q\in \mathbb Q].$$ In the above each $\X_q$ is an indeterminate, thus $\hat{\mathbb{A}}$ is a multivariable polynomial ring. 
The ordering on $\hat{\mathbb{A}}$ is obtained by treating each $\X_q$ as infinite, and stipulating that each $\X_q$ dominates all polynomial 
expressions whose indeterminates come from $\{\X_{q'}: q'<q\}$.  The model $\mathbb A$ was introduced in \cite{enlel:cat24}, where it was 
shown to satisfy $\mathsf{PA^-+Coll}.$\footnote{It has been known for some time that 
$\mathbb{Z}[\X_{\alpha}\mathop\mid \alpha\in \omega_1]^{\mathsmaller{\geq 0}}$ is a model $\mathsf{PA^-+Coll}$ in 
which the standard cut is definable (here $\omega_1$ be the first uncountable ordinal). See Exercise 7.7 of Kaye's 
textbook \cite[Chapter 7.2]{kaye:mode91}. However, in contrast to $\mathbb{A}$, this model cannot be interpreted in $\mathbb{N}$ since it is uncountable.} 

We present an alternative description of $\hat{\mathbb{A}}$ using the model-theoretic notion of \emph{directed union}.\footnote{The directed union construction is a special 
case of the \emph{direct limit}, in which the embeddings that ``glue" the constituent pieces of the model together are inclusion maps.} 
This might strike the reader as spurious duplication, but we 
find it useful in developing a `geometric' intuition about the model, and also because it 
provides a dress rehearsal for the analysis of the far more complicated model $\mathbb{B}$ of $\mathsf{IOpen+Coll}$ presented in Section \ref{iterated_shepherdson}.
 
 Recall that a partially ordered set $\mathbb{D}$ is said to be \emph{directed} if any pair of elements of $\mathbb{D}$ have a common upper bound. Given a structure $\mathcal{M}$, a directed set $\mathbb{D}$, and an indexed collection $\{\mathcal{M}_d: d\in \mathbb{D}\}$ of submodels of $\mathcal{M}$, $\mathcal{M}$ is said to be a \emph{directed union} of $\{\mathcal{M}_d: d\in \mathbb{B}\}$, if 
 \begin{enumerate}[$1$.]
 \item $\mathcal{M=}\bigcup\limits_{d\in \mathbb{D}}\mathcal{M}_{d}$, and
 \item $\mathcal{M}_{d}$
is a submodel of $\mathcal{M}_{d^{\prime }}$ for $d\leq _{\mathbb{D}}d^{\prime }.$ 
\end{enumerate}

 Given an ordered ring $\mc R$ and an indeterminate $\X$, we use $\mc R[\X]$ to refer to the ordered ring of polynomials with coefficients in $\mc R$, where $\X$ is greater than all elements of $\mc R$, precisely as in the special case where $\mc R=\mathbb{Z}$, discussed in Subsection \ref{secondstandard}.
 
 Let $[\mathbb{Q}]^{<\omega}$ be the directed set of finite subsets of $\mathbb{Q}$ (ordered by inclusion). For $\fsvar\in [\mathbb{Q}]^{<\omega}$, the ordered ring $\hat{\mathbb{A}}(\fsvar)$ is constructed by the following clauses: 
  \begin{itemize}
 \item $\hat{\mathbb{A}}({\varnothing}):=\mathbb{Z}$, and 
 \item $\hat{\mathbb{A}}({\fsvar\cup\{q\}}):=\hat{\mathbb{A}}(\fsvar)[\X_q]$, where $q'<q$ for all $q'\in \fsvar.$
 \end{itemize}

 Officially speaking, $$\hat{\mathbb{A}}:=\bigcup\limits_{\fsvar\in [\mathbb{Q}]^{<\omega}}\hat{\mathbb{A}}(\fsvar).$$ 
 Clearly, $\hat{\mathbb{A}}(\fsvar)$ is a submodel of $\hat{\mathbb{A}}(\fsvar')$ for $\fsvar\subseteq \fsvar'$. Coupled with the fact $[\mathbb{Q}]^{<\omega}$ is a directed set under containment, we have the following observation:
 
 \begin{lem} \label{hatA-directed_union} $\hat{\mathbb{A}}$ is a directed union of its constituent submodels $\hat{\mathbb{A}}(\fsvar)$, as $\fsvar$ ranges in $ [\mathbb{Q}]^{<\omega}$.
 \end{lem}
 
 Therefore, based on classical facts about direct limits (as in \cite[Section 2.3]{hodg:mode93}), if $\varphi$ is a sentence in the language of ordered rings that can be written in the form 
 $\forall\exists$, and $\varphi$ holds in $\hat{\mathbb{A}}(\fsvar)$ for all $\fsvar\in [\mathbb{Q}]^{<\omega}$, 
 then $\varphi$ holds in $\hat{\mathbb{A}}$ as well. This makes it clear that $\hat{\mathbb{A}}$ is a discretely ordered ring, which in turn shows that its non-negative part, i.e., $\mathbb{A}$, satisfies $\mathsf{PA^-}$.\footnote{It is straightforward to verify that if all the constituent models $\mathcal{M}_d$ of a directed union $\mathcal{M}$ satisfy a sentence $\varphi$ that can be put in the form $\forall \exists$, then $\mathcal{M}\models \varphi.$ The classical Chang-Suzko theorem shows that the converse is also true, but we don't need this fact.}
 
\begin{definition}The \emph{support} of an element $p\in\hat{\mathbb{A}}$, written $\mathsf{supp}(p)$, is the smallest $\fsvar\in[\mathbb{Q}]^{<\omega}$ (in the sense of inclusion) such that $p\in \hat{\mathbb{A}}(\fsvar)$. Standard arguments show that $\mathsf{supp}(p)$ is well-defined (see Section \ref{Section_on_supports} of the Appendix).
\end{definition}

It will be useful to have names for certain subrings of  $\hat{\mathbb{A}}$, and certain initial segments of $\mathbb{A}$. 

\begin{definition}
    In what follows, given $q\in\mathbb{Q}$, we take $\mathbb{Q}_q:=\{q'\in \mathbb{Q}\mathop\mid q'<q\}$.
 \begin{itemize}
\item $\hat{\mathbb{A}}_{q}:=\mathbb {Z}[\X_{q'}\mathop\mid q'\in \mathbb Q_{q}]~=\bigcup\limits_{\fsvar\in [\mathbb{Q}_{q}]^{<\omega}}\hat{\mathbb{A}}(\fsvar)$, and
\item $\mathbb{A}_{q}:=\hat{\mathbb{A}}_{q}^{\mathsmaller{\geq 0}}.$
\end{itemize}
\end{definition}
It can be readily checked that $\mathbb{A}_{q}\subsetneq_{\mathrm{end}} \mathbb{A}_{q'}$ whenever $q<q'$. Indeed, as shown in Theorem \ref{eee_for_A}, $\mathbb{A}_{q}\prec_{\mathrm{end}} \mathbb{A}$ for each $q\in \mathbb{Q}$, which implies that $\mathbb{A}_{q}\prec_{\mathrm{end}} \mathbb{A}_{q'}$ for $q<q'$. 

Note that for $q<q'$ we have: $$\mathbb{A}_q\subsetneq_{\mathrm{end}} \hat{\mathbb{A}}_q [\X_q]^{\mathsmaller{\geq 0}}\subsetneq_{\mathrm{end}} \mathbb{A}_{q'}.$$ 

We also observe that $\{\X_{q}^{n} \mathop\mid n\in \mathbb{N}\}$ is cofinal in $\hat{\mathbb{A}}_q [\X_q]^{\mathsmaller{\geq 0}}$, and $\{\X_{q}-n \mathop\mid n\in \mathbb{N}\}$ is downward cofinal in $\hat{\mathbb{A}}_q [\X_q]^{\mathsmaller{\geq 0}} \setminus \mathbb{A}_q$.

\subsection{$\mathbb A$ is a model of $\pam$ plus full Collection}
We have already explained why $\mathbb{A}$ satisfies $\mathsf{PA^-}$. We now verify that $\mathbb{A}$ also satisfies
{\sf Coll}. This result was first proved in \cite[Corollary 34]{enlel:cat24}.

Recall that for $q\in \mathbb{Q}$, $\mathbb A_q$ is the initial segment of $\mathbb{A}$ that consists of elements whose support is a subset of  $\{\X_{q'} \mathop\mid q'<q\}$.

\begin{theorem}\label{eee_for_A}
   For each  $q\in \mathbb{Q}$ the model $\mathbb A_q$ is a downwards closed cut of $\mathbb A$ which is closed under the arithmetical
   operations. Moreover, $\mathbb A_q \prec_{\mathrm{end}} \mathbb{A}$.
\end{theorem}

\begin{proof} 
We only prove $\mathbb A_q \prec \mathbb{A}$ 
 since it is straightforward to verify that $\mathbb A_q$ is a downwards closed cut of $\mathbb A$. Consider a formula 
 $\varphi (p_{0},\cdots ,p_{k-1})$, where $p_{0},\cdots ,p_{k-1}$ are parameters from $\mathbb{A}_q.$ Since the support of each 
 polynomial is finite, there is some $q'<q$ such that $\mathbb{Q}_{q'}$ contains the support of each $p_i$. By  a standard argument there is an order 
isomorphism $f:\mathbb{Q}_{q}\rightarrow \mathbb{Q}$ such that $f(x)=x$ for each $x<q'$. This isomorphism naturally induces an
isomorphism $F:\mathbb{A}_q \rightarrow 
\mathbb{A}$ such that $F$ fixes each element of $\mathbb{A}_{q'}$. 
Since the isomorphism $F$ fixes each $p_i$, this makes it clear that 
$\mathbb{A}_{q}\models \varphi (p_{0},\cdots ,p_{k-1})~\mathrm{iff~}
\mathbb{A}\models \varphi (p_{0},\cdot \cdot \cdot,p_{k-1})$. 
\end{proof}

\begin{theorem} \label{coll_in_A}
    $\mathbb A$ satisfies {\sf Coll}.
\end{theorem}

\begin{proof} 
Consider any $q \in \mathbb Q$. By Theorem \ref{eee_for_A}, (\dag) $\mathbb{A}_{q}\prec_{\mathrm{end}} \mathbb{A}$.
So it suffices to show that $\mathbb{A}_{q}$ satisfies {\sf Coll}.
Suppose we have $\mathbb{A}_q\models\forall x \bles a \,\exists y~\phi(x,y,\vv b)$. By (\dag), we find
$\mathbb{A}\models\forall x \bles a \,\exists y\bles \X_q ~\phi(x,y,\vv b)$. So, again by (\dag),
$\mathbb{A}_q\models\exists w\, \forall x \bles a \,\exists y\bles w~\phi(x,y,\vv b)$.
 \end{proof}

\begin{rem} \label{eee_and_coll} The reasoning used in the proof of Theorem \ref{coll_in_A} shows, more generally, that 
if $\mathcal{M}$ is an ordered structure that has an elementary initial segment \textup(or an elementary end extension\textup), 
then the collection scheme holds in $\mathcal{M}$.\footnote{This fact is well-known and due to Keisler, 
who also showed, using an omitting types theorem, that if $\mathcal{M}$ is a countable ordered structure 
in a countable language (signature) such that $\mathcal{M}$ satisfies the collection scheme, 
then $\mathcal{M}$ has an elementary end extension.} 
\end{rem}
    
\subsection{Standard forms and pseudo-variables}
We need the notion of \emph{standard form} for elements of $\hat{\mathbb{A}}$. We say that the expression $$a_n\X_{q_n}^{k_n} + a_{n-1}\X_{q_{n-1}}^{k_{n-1}} + \dots +a_0$$ is the standard form of $x\in\hat{\mathbb{A}}$, if $x=a_n\X_{q_n}^{k_n} + a_{n-1}\X_{q_{n-1}}^{k_{n-1}} + \dots +a_0$, and the following conditions are satisfied (note that (5) is a recursive condition). 

\begin{enumerate}[$1$.]
\item $a_i \neq 0$ for $0<i\leq n$.
\item $k_i\in \mathbb{N}^{\mathsmaller{>0}}$ for $0<i\leq n$.
\item
$a_0 \in \mathbb{Z}$. 
\item $a_i\in \hat{\mathsf{A}}_{q_i}$ for $0<i\leq n$.
 \item  $a_i$ is itself in standard form for $0<i\leq n$.
 \item  $q_{i+1} \geq q_i$ and, if
$q_{i+1}=q_i$, then $k_{i+1} > k_i$; here $0<i<n.$
\end{enumerate}

    A standard form expression $a_n\X_{q_n}^{k_n} + a_{n-1}\X_{q_{n-1}}^{k_{n-1}} + \dots +a_0$
    \emph{is the standard form of an element of} $\mathbb{A}$ if $a_n\geq 0$.

\noindent
Observe that if $x\in\hat{\mathbb{A}}$, then ${\sf supp}(x)$ consists of the $\X_q$ such that $\X_q$ occurs in the standard form of $x$. Also, note that if $x$ divides $x'$ in $\hat{\mathbb A}$, then ${\sf supp}(x)\subseteq {\sf supp}(x')$. 
This observation immediately shows that we have unique factorisation in prime factors in $\hat{\mathbb A}$,
since factorisation in the case of $y$ reduces to factorisation in $\mathbb Z[{\sf supp}(y)]$.

Suppose $a_n\X_{q_n}^{k_n}  + \dots a_0$ is the standard form of $x\in \mathbb{A}$. 

We define:
\begin{itemize}
\item
${\sf rank}_0(x) := q_n$, if $n>0$ and ${\sf rank}_0(x) := -\infty$ if $n=0$. .
\item
${\sf rank}(x) := (q_n,k_n)$, if 
$n>0$ and ${\sf rank}_0(x) := (-\infty,0)$ if $n=0$.\footnote{Here we use $(x,y)$ for the formation of the ordered pair of $x$ and $y$ in the metalanguage.}
\item
$(q,k) < (r,m)$ iff $q<r$ or ($q=r$ and $k<m$).
\end{itemize} 
 
 We say $y\X_q^k+z$ is the \emph{basic form} of $x\in\mathbb{A}$, if 
 $x=y\X_q^k+z$ and ${\sf rank}_0(y) < q$ and ${\sf rank}(|z|)<(q,k)$.

We now give intuitions about the above definitions. Observe that for $x\in\mathbb{A}$, ${\sf rank}_0(x) := q$ iff $q$ is the 
least element of $$\{q'\in \mathbb{Q}  \mathop\mid x\in \hat{\mathbb{A}}_{q'} [\X_{q'}]^{\mathsmaller{\geq 0}}\}.$$  
Thus, if ${\sf rank}_0(x) := q$, then $x$ can be written in the basic form $y\X_q^k+z$, where 
$y\in \mathbb{A}_{q}$, and $|z|$ is a polynomial in $\mathbb{A}_q [\X_q]$ of degree less than $k$. 

We observe that $\hat{\mathbb{A}}$ is a Dorroh ring, as witnessed by the ideal consisting of $x\in \hat{\mathbb{A}}$ such that $a_0 = 0$, where $x$ is written in standard form. 
Thus, $\omega$ is definable in $\mathbb{A}$. For $x$ and $y$ in $\mathbb{A}$ let us write:
\begin{itemize}
\item
  $y \ll x$ for $\forall n\bin \omega \; (y+2)^n< x$.
\item
$\psv(x) :\iff \forall y \bleq x ~ (y\ll x\vee(x-y) \ll x).$
\end{itemize}
We refer to the elements of $\mathbb{A}$ satisfying $\psv$ as the \emph{\psvar}. This terminology is motivated by 
Theorem \ref{autosmurf} which shows that the collection of pseudovariables `behaves like' the collection of indeterminates in $\mathbb{A}$.

\begin{theorem}
For all $x$ and $y$ in $\mathbb{A}$ we have: $y \ll x$ iff ${\sf rank}_0(y) < {\sf rank}_0(x)$.
\end{theorem} 

\begin{proof}
Let $q^\ast := {\sf rank}_0(x)$ and
$r^\ast:= {\sf rank}_0(y)$. 
It follows that $x$ has basic form $z\X^n_{q^\ast}+w$, 
and $y$ has basic form $u\X^m_{r^\ast}+v$.

Suppose $y \ll x$.
Let $mk > n$.
We find $y^k < x$, so $r^\ast < q^\ast$.

Conversely, if $r^\ast < q^\ast$, it is easy to see that for all $n$, $y^n < x$.
\end{proof}

\begin{theorem}\label{gargamel}
For each $x\in\mathbb{A}$, $\psv(x)$ holds\ iff $x= \X_q+y$, where $\X_q+y$ is in basic form 
\textup(so, $|y| \ll \X_q$\textup).
\end{theorem}

\begin{proof}
Clearly, no element of $\omega_{\mathbb{A}}$ satisfies $\psv$. 

Suppose $x$ has basic form $z\X_q^n+y$.
If $z>1$ or $n>1$, neither $\X_q\ll x$ nor $x-\X_q \ll x$.
So, if $x$ is in
 \psv, we must have $x= \X_q+y$, where $|y|\ll \X_q$.

Conversely, suppose $x = \X_q+y$ 
Suppose $z \not\ll x$. Then, $z^m > x$, for some $m$. It follows that $z$ must have basic form $u\X_r^k+v$,
where $r\geq q$. If $u>1$ or $r>q$ or $k>1$, we find $x-z=0$ and we are done. Suppose $z=\X_q+v$. 
It follows that $x-z = y-v \leq y \ll x$.
\end{proof}

\subsection{Interpreting $\mathbb N$ in $\mathbb A$}
Recall that $\mathbb{A}=\hat{\mathbb{A}}^{\mathsmaller{\geq 0}}$. Since $\hat{\mathbb{A}}$ is clearly an ordered Dorroh ring, the standard cut $\omega_{\mbbs A}$ of $\mathbb A$ is definable in $\mathbb A$ by Theorem \ref{Dorroh_fourquares}. This gives us an interpretation, say {\sf N}, of $\mathbb N$ in $\mathbb A$. Note that {\sf N} is just domain restriction to $\omega_{\mbbs A}$.

\subsection{Interpreting $\mathbb A$ in $\mathbb N$}\label{translation_alpha}
Inside $\mathbb N$, we can build a copy of $\mathbb A$ in some
standard way. This gives rise to an interpretation {\sf A} of $\mathbb A$ in $\mathbb N$.

We use $\alpha$ for the underlying translation of ${\sf A} \circ {\sf N}$.
Thus, we  will call the internal copy $\mathbb A^\alpha$. 

We will use $\mf F$ for the canonical isomorphism from $\mathbb A^\alpha$ to $\mathbb A$. We note that
this isomorphism is not $\mathbb A$-internally definable. 

\subsection{Automorphisms}
In this section, we acquire some information on automorphisms of $\mathbb A$.

\begin{theorem}\label{autosmurf}
Suppose $x_0 \ll x_1 \ll \dots \ll x_{n-1}$ and $y_0 \ll y_1 \ll \dots \ll y_{n-1}$, where the $x_i$ and $y_i$  
are elements of $\mathbb{A}$ that are \psvar. Then, there is an $\mathbb A$-automorphism $\sigma$ that sends $x_i$ to $y_i$.
\end{theorem}

\begin{proof}
Clearly, it is sufficient to prove this for the case that $x_i = \X_{q_i}$, where $q_i<q_j$ if $i<j$. 
Moreover, if we have $q_0 < q_1 \dots < q_{n-1}$ and $r_0 < r_1 \dots < r_{n-1}$, then there is
an order-automorphism $\theta$ of $\mathbb Q$ such that $\theta(r_i)=q_i$. 
So, we have an automorphism $\tau$ of $\mathbb A$ that moves $\X_{r_i}$ to $\X_{q_i}$.
By Theorem~\ref{gargamel}, it follows that it suffices to consider the case where $x_i= \X_{q_i}$ and $y_i = \X_{q_i}+a_i$.
We assume this simplification.

We take $\sigma(z) := z[x_0:y_0,\dots,x_{n-1}:y_{n-1}]$. Here, of course, $\sigma$ operates identically on all 
variables distinct from the $x_i$. This is clearly a morphism. It suffices to define
an inverse morphism.

We define the inverse  $\mu$ by:
\begin{itemize}
\item
$\mu$ operates identically on all variables not among the $\X_{q_i}$.
\item
$\mu(\X_{q_i}) := \X_{q_i} - \mu(a_i)$, for $i<n$.\\
We note that this is not circular since $a_i$ only contains
variables of lower index.
\end{itemize}
 We now have, assuming that we have the desired result already for $j<i$:
 \begin{eqnarray*}
 \sigma\mu (\X_{q_i}) & = &  \sigma(\X_{q_i} - \mu(a_i)) \\
 & = & \sigma(\X_{q_i}) - \sigma\mu(a_i) \\
 & = & \X_{q_i}+a_i -a_i \\
 &= & \X_{q_i}
 \end{eqnarray*}
 \text{and} \\
 \begin{eqnarray*}
  \mu\sigma (\X_{q_i}) & = &  \mu(\X_{q_i} + a_i) \\
  & = & \mu (\X_{q_i}) + \mu(a_i) \\
  & = & \X_{q_i}-\mu(a_i)+\mu(a_i) \\
  & = & \X_{q_i}.
 \end{eqnarray*}
 We note that we used our assumption that we have the result for $j<i$ to conclude  that $\sigma\mu(a_i) = a_i$.
\end{proof}

\begin{question}\label{Question_Aut(A)}
    Is there a characterization of the group $\mathrm{Aut}(\mathbb{A})$ of all automorphisms of $\mathbb{A}$?
\end{question}

\subsection{Extension Lemma}\label{sect_ext}
In this section, we prove an important lemma that will help to verify the
correctness of the  Ehrenfeucht-Fra\"{\i}ss\'{e} game that we will specify.

For a subset $W$ of $\mathbb{A}$ consisting of pseudo-variables, we say that $W$ is \emph{independent} when no two elements of $W$ are finitely apart, i.e., when no two elements of $W$ have the same $\mathsf{rank}_0$. We define the support of a set as the union of the supports of its elements.

\begin{lem}\label{spelendesmurf} Let $W$ be an independent finite set of \psvar. Then, there is a finite set of variables $V$ so that
 $W\cup V$ is independent and ${\sf supp}(W)$ is polynomially definable from $W\cup V$.
 \end{lem}

\begin{proof}
Consider an independent finite subset of \psvar\ $W$. Let
$W$ be $\verz{w_0, \dots w_{k-1}}$, where ${\sf rank}_0(w_i) < {\sf rank}_0(w_{i+1})$.
Let \[W^\star:= \verz{\X_{q_i}\mathop\mid {\sf rank}_0(w_i)=q_i\text{ and } i<k}.\]
Let 
$V :=  {\sf supp}(W) \setminus W^\star$ and let $Z := W\cup V$. We note that $Z$ is independent.

We have to show that the elements of ${\sf supp}(W)$ are polynomially definable from $Z$.
It clearly suffices to show that the elements of $W^\star$ are polynomially definable from $Z$.
Let these elements be $\X_{q_1},\dots,\X_{q_{k-1}}$.
Suppose, for each $j<i$, the element $\X_{q_i}$ is polynomially definable from $Z$.
We have $w_i = \X_{q_i}+b_i$, where $b_i$ contains at most elements of ${\sf supp}(W) \setminus W^\star$ plus
elements $\X_{q_j}$, for $j<i$. All these are definable from $Z$ by our assumption.
This tells us that $b_i$ and, hence, $\X_{q_i}$ is definable from $Z$.
By course of values induction, we are done.
\end{proof}

\begin{theorem}\label{smurfoludens}
Suppose $W$ is an independent finite set of \psvar\ in $\mathbb{A}$
and suppose $x$ is any element in $\mathbb{A}$. 
Then, there is an independent finite set of \psvar\ $U$  in $\mathbb{A}$ such that
$W\subseteq U$ and $x$ is polynomially definable from $U$.
\end{theorem}
\begin{proof}
Let $W$ be an independent finite set of \psvar\ in $\mathbb{A}$ and let $x$ be any element of $\mathbb{A}$. 
Let $V$ be as promised in Lemma~\ref{spelendesmurf} for $W$. 
We take $U := W \cup V \cup ({\sf supp}(x) \setminus {\sf supp}(W))$.  Since $V$ consists of variables, 
$U$ is independent and, since ${\sf supp}(W)$ is polynomially definable from $W\cup V$,
 every element of ${\sf supp}(x)$ is polynomially definable from $U$.
\end{proof}
\noindent
We note that Theorem~\ref{smurfoludens} is a first-order statement in $\mathbb A$.

\subsection{An Ehrenfeucht-Fra\"{\i}ss\'{e} Game}\label{sect_pam_ehren}
Recall from Subsection \ref{translation_alpha} that $\alpha$ is the translation that delivers the $\mathbb{A}$-internal copy $\mathbb{A}^{\alpha}$ of $\mathbb{A}$. Clearly, $\mathbb A\models \phi \iff \phi^\alpha$, for all sentences $\phi$. 
However, we want to make 
 the argument for this fact $\mathbb A$-internal, in such a way that we can extract a finite subtheory $\Phi$ of
${\sf Th}(\mathbb A)$, such $\Phi \vdash \phi \iff \phi^\alpha$, for all sentences $\phi$. We note that
the isomorphism, say $\mf F$, from $\mathbb A^\alpha$ to $\mathbb A$ is not $\mathbb A$-definable,
so we need something else. We develop an Ehrenfeucht-Fra\"{\i}ss\'{e} game between $\mathbb A^\alpha$ and
$\mathbb A$.

In this subsection, we will use $x,y,\dots$ for general $\mathbb A$-elements and $u,v,\dots$ for
$\mathbb A^\alpha$-elements

We assume that the $\mathbb A^\alpha$-elements $u$ are represented by (something like)
numerical codes of their canonical representations by polynomials. 
We note that $\mathbb A^\alpha$ is entirely constructed inside the internal copy of $\omega$.
Inside $\omega$ we can, of course, also define things like the support function for 
the elements of $\mathbb A^\alpha$ qua coded polynomials, even if the real support function for
$\mathbb A$ is not $\mathbb A$-definable (since $\mathbb{A}$ has lots of automorphisms that send variables to pseudo-variables). 

We will think of $\alpha$ as a translation that also works
on an extended signature for $\mathbb A$ which includes the variables and the support function.
Thus, we will write $\X_q^\alpha$ and ${\sf supp}^\alpha$ for the $\mathbb A^\alpha$-internal variables
and the $\mathbb A^\alpha$-support function.

We define the promised Ehrenfeucht-Fra\"{\i}ss\'{e} game.
We say that a finite function $F$ is a \emph{witnessing function} if it has as domain a finite set of the $\X^\alpha_q$ and as range a
set of \psvar\ with the constraint that, if $\X^\alpha_q$ and $\X^\alpha_{q'}$ are in the domain and if $q<q'$, then
$F(\X^\alpha_q) \ll F(\X^\alpha_{q'})$.

Clearly, we can define ${\sf ev}_F(u)$ a function that evaluates $u$ qua coded
canonical polynomial by giving the variables $\X^\alpha_q$ the evaluation $F(\X^\alpha_q)$.

Suppose $G$ is a finite function from $\mathbb A^\alpha$ to $\mathbb A$.
We say that $G$ is \emph{good} if, there is a witnessing function $F$, such that
${\sf dom}(F) \supseteq {\sf supp}^\alpha({\sf dom}(G))$ and, for all $u \in {\sf dom}(G)$,
we have $G(u) = {\sf ev}_F(u)$.

Suppose $P(\vv x)$ is one of $x=y$, $x\leq y$, ${\sf Z}(x)$, ${\sf S}(x,y)$, ${\sf A}(x,y,z)$, ${\sf M}(x,y,z)$.
Here, e.g., {\sf A} is the relational version of addition. We have:

\begin{lem}[Atomic Clause]
Consider the formula $P$ and a sequence $\vv u$ whose length equals the arity of $P$. In $\mathbb A$, we have the following. Suppose $G$ is a good function and the domain of
$G$ extends $\vv u$.
Then, we have $P^\alpha(\vv u) \iff P(G(\vv u))$. 
\end{lem}

\begin{proof}
Let $F$ witness the goodness of $G$.
By Theorem~\ref{autosmurf}, there is an $\mathbb A$-automorphism $\sigma$ that sends  $F(\X_q^\alpha)$ to $\X_q$, for 
$\X_q^\alpha$ in the domain of $F$. 
We have:  \qedright
\begin{eqnarray*}
P^\alpha(\vv u) & \iff & P(\mf F(\vv u)) \\ 
& \iff  & P({\sf ev}_{\mf F}(\vv u)) \\
& \iff  & P({\sf ev}_{\sigma\circ F}(\vv u)) \\
& \iff  & P(\sigma({\sf ev}_{F}(\vv u))) \\
& \iff & P({\sf ev}_F(\vv u)) \\
& \iff & P(G(\vv u))
\end{eqnarray*}
\end{proof}

\begin{lem}[Forward Clause]
Let $G$ be a good function and let $u$ be an element of $\mathbb A^\alpha$.
Then, there is a good function $H \supseteq G$, such that $u$ is in the domain of
$H$.
\end{lem}

\begin{proof}
Let $F$ be a witnessing function for $G$ with domain $V$. Let
$W:= V \cup\, {\sf supp}^\alpha(u)$ given by $w_0\ll^\alpha \dots \ll^\alpha w_{m-1}$. We can effectively find rational numbers
$q_0 <\dots < q_{m-1}$ such that, if $w_j\in V$, then  ${\sf rank}_0(F(w_j)) = q_j$. We set $F^\star(w_j) := F(w_j)$, if
$w_j\in V$ and $F^\star(w_j) = \X_{q_j}$, otherwise. 
Finally, we take
$H := G \cup \verz{\tupel{u, {\sf ev}_{F^\star}(u)}}$.
\end{proof}

\begin{lem}[Backwards Clause]
Let $G$ be a good function and let $x$ be an element of $\mathbb A$.
Then, there is a good function $H \supseteq G$, such that $x$ is in the range of
$H$.
\end{lem}

\begin{proof}
Let $F$ witness that $G$ is good and let $X$ be the range of $F$. 
By Theorem~\ref{smurfoludens}, there is an independent set of 
\psvar\ $Y\supseteq X$, such that $x$ is polynomially definable from $Y$.
Let $Y$ be given by $y_0 \ll \dots \ll y_{m-1}$. We can find rational numbers $q_0<\dots < q_{m-1}$, such that, if
$y_j \in X$, then $F(\X^\alpha_{q_j}) = y_j$. We set $F^\star(\X^\alpha_{q_j}) := y_j$. Let $u$ be a code of the polynomial defining $x$ from $W$ employing the variables $\X_{q_i}$. Then, we can set $H = G\cup \verz{\tupel{u,x}}$.
The goodness of $H$ is witnessed by $F^\star$. 
\end{proof}

By standard arguments, we have the following lemma.
\begin{lem} \label{finite_witnessing}Let $\Phi$ be given as the conjunction of \pam, $({\pam})^\alpha$, and the atomic, backwards and forwards clauses. Then, $\Phi \vdash \phi\iff \phi^\alpha$. 
\end{lem}

\subsection{Completing \pam\ plus Collection}
We  apply Corollary~\ref{cor_mod_retract} with:
\begin{itemize}
    \item 
$\mathbb A$ in the role of $\mc M$, 
\item 
$\leq$ in the role of $\belowz$, 
\item 
$\mathbb N$ in the role of
$\mc N$, 
\item 
{\sf N} in the role of $\imath$,
\item 
{\sf A} in the role of $\jmath$. 
\item
$\omega_{\mbbs A}^{\sf c}$ in the role of $\theta$.
\item 
The finite statement $\Phi$ (as in Lemma \ref{finite_witnessing}) of the atomic and back and forth clauses
for the good functions together with \pam\ and ${\pam}^\alpha$  delivers 
our desired witness of a ${\sf INT}^{\sf fw}_3$-retract.
\end{itemize}

It is easy to see that our choices realise the conditions of the Corollary \ref{cor_mod_retract}.
We note that $\tau_{\sf N}$ and $\omega_{\mbbs A}^{\sf c}$ are $\Delta_0$.
So it follows that we have a sentence $\Phi$, and a set $\mc X$
of $\Delta_1$-formulas  (over $\Phi$), such 
that $\Phi + \mc X$ axiomatises  $\mathsf{Th}(\mathbb{A})$. 

We conclude:

\begin{theorem}\label{uitslagcollsmurf}
    The theory $\pam+{\sf Coll}$ has a completion that is axiomatised by a sentence $\Phi$ plus a set of $\Sigma_1$-sentences. 
\end{theorem}

\subsection{Free Riders}
We note that, by a compactness argument, the ordered ring $\hat{\mathbb A}$ is a Unique Factorisation Domain. Hence, our model $\mathbb A$ of $\pam+{\sf Coll}$ has additional salient properties, like the GCD property (existence of greatest common divisors). However, the GCD property is not provable in $\pam+{\sf Coll}$.\footnote{Indeed the GCD property is not even provable in $\pam+{\sf Coll}$ since the iterated Shepherdson model constructed in the next section (which satisfies $\mathsf{Coll}$) inherits certain algebraic `defects' of the Shepherdson model, including the failure of factorization into irreducibles.} 
To see this, let $\mathcal{R}$ be the analogue of $\hat{\mathbb A}$, where we allow coefficients in $\mathbb Q$ rather than in $\mathbb{Z}$, except for the case
of the coefficient with no variables. It is evident that $\mathcal{R}^{\mathsmaller{\geq 0}}$ is a model of $\pam$. Our method of verification of $\mathsf{Coll}$ in $\mathbb{A}$ can be used here as well to show that $\mathcal{R}^{\mathsmaller{\geq 0}}$ satisfies $\mathsf{Coll}$. However, $\mathcal{R}$ is not even a factorization domain (let alone a unique factorization domain) and therefore it does not have the GCD property. This is based on the following observation: any $\X_q$ can be written as the product of $(1/2)\X_q$ and $2$ in $\mathcal{R}$, neither of which is a unit (invertible). Another way to see that the GCD property fails in $\mathcal{R}$ is to note that the elements $\X_0$ and $\X_1$ do not have a greatest common divisor in $\mathcal{R}$.

\section{A restricted Completion of {\sf IOpen} plus full Collection}\label{iterated_shepherdson}

In this section, we construct a recursive model, denoted $\QS$, of $\mathsf{IOpen + Coll}$ whose 
complete theory is axiomatized by a collection of axioms consisting of a single 
sentence together with an infinite number of $\Sigma_1$-sentences. 
Intuitively speaking, the relationship of $\QS$ to $\mathbb{S}$
(the Shepherdson model of {\sf IOpen}) is that of 
 the model $\mathbb{A}$ of the previous section to the model $\mathbb Z[\X]^{\mathsmaller{\geq 0}}$. For this reason we sometimes refer to $\mathbb{B}$ as the $\mathbb{Q}$\emph{-iterated Shepherdson model}.\footnote{As we shall see, $\mathbb{B}$ has the `fractal property' of being is isomorphic to many submodels of itself, some of which are proper initial segments of $\mathbb{B}.$} Given that ${\sf IOpen}$ is a much stronger theory than ${\sf PA}^-$, the model $\QS$ is far more complicated than $\mathbb{A}$.

\subsection{Basic properties of the $\mathbb{Q}$-iterated Shepherdson model}  
Assume that $\mc R$ is an arbitrary discretely ordered commutative ring. $\mathsf{rc}(\mc R)$ denotes the real closure of the fraction field of $\mc R$. See Subsection \ref{realsmurf} for the details of the construction. Note that we assume that $\mc R$ is a subring of ${\sf rc}(\mc R)$. Given an indeterminate $\X$ (for our purposes $\X$ will be of the form $\X_q$ for some $q \in \mathbb{Q}$), the \emph{Shepherdson extension of $\mc R$}, denoted $\mathcal{S}(\mc R,\X)$ is the ordered commutative subring of $\mathsf{rc}(\mc R[\X])$ that consists of fractional polynomials of the following form, often referred to as Puiseux polynomials:
    \[a_n\X^{n/q}+a_{n-1}\X^{n-1/q} + \ldots + a_1 \X^{1/q} + a_0,\]
    where 
    \begin{enumerate}[$1$.]
        \item    $n,q\in \mathbb{N}$, $q>0$.
        \item $a_0\in \mc R$; and
        \item $a_i\in \mathsf{rc}(\mc R)$ for $1\leq i\leq n.$
       
    \end{enumerate} 
    The order is determined by making $\X$ infinite relative to the elements of $\mc R$. 
Thus, $\shef(\mathbb{Z},\X)$ is the Shepherdson model $\mathbb{S}$ of $\mathsf{IOpen}$ we considered in Section \ref{openitself}.
We have the following basic lemma.
\begin{lem}\label{kleinesmurf}
Given a discretely ordered commutative ring $\mc R$, there is a unique isomorphism of ${\sf rc}(\shef(\mc R,\X))$ to ${\sf rc}(\mc R[\X])$ over
$\mc R[\X]$.
\end{lem}

\begin{proof} This follows from Theorem \ref{rc_is_unique},
since $\mc R[\X]$ is a subring of $\shef(\mc R,\X)$, and $\shef(\mc R,\X)$ is a subring of $\mathsf{rc}(\mc R[\X])$.
\end{proof}

The following theorem is needed in the proof of Theorem \ref{Marker_thm}. Recall that a subring $\mc R$ of an ordered field $\mc K$ is 
said to be an \emph{integer part} of $\mc K$ if for all $k\in \mc K$ there is $r\in \mc R$ such that $\left\vert r-k\right\vert <1$. Let us mention that 
Shepherdson established the equivalence of (1) and (2). The equivalence of (2) and (3) follows from the well-known fact that if 
$\mc R$ is an ordered subring of a real closed field $\mc K$, then there is a (unique) order preserving field embedding 
$e:\mathsf{rc}(\mc R)\rightarrow \mc K$ such that $e(x)=x$ for all $x\in \mc R$.

\begin{theorem} [Shepherdson \cite{shep:nons64}] \label{Shepherdson_thm} The following are equivalent for a discrete ordered commutative ring $\mc R$:
\begin{enumerate}[$1$.]
    \item  $\mc R^{\mathsmaller{\geq 0}}\models {\sf IOpen}$.
    \item $\mc R$ is an integer part of $\mathsf{rc}(\mc R).$
    \item $\mc R$ is an integer part of a real closed field.
\end{enumerate}
\end{theorem}

\begin{theorem}[Marker \cite{markerIOpen}] \label{Marker_thm}
    If $\mc R^{\mathsmaller{\geq 0}}\models {\sf IOpen}$, then $\shef(\mc R,\X)^{\mathsmaller{\geq 0}}\models {\sf IOpen}$.\footnote{Note that this theorem shows that every model of $\mathsf{IOpen}$ (of any cardinality) has a proper end extension to a model of $\mathsf{IOpen}$.}
\end{theorem}
\begin{proof}
  Let $\mc R$ be a discrete ordered commutative ring such that $\mc R^{\mathsmaller{\geq 0}}\models {\sf IOpen}$, and let $\mc K=\mathsf{rc}(\mc R)$. Note that by Theorem \ref{Shepherdson_thm} $\mc R$ is an integer part of $\mc K$. Let $\mc K^{*}$ be the \emph{Puiseux extension} of $\mc K$, i.e., 
 $$\mc K^{*}:= 
\bigcup\limits_{n\in \mathbb{N}^{>0}}\mc K(\!(\X^{1/n})\!),$$
where $\X$ is infinite with respect to $\mc K$, and $\mc K(\!(\Y)\!)$ is the field of formal Laurent series in the indeterminate $\Y$. Thus elements of $\mc K^{*}$ are formal infinite series of the form 
\[k^{*}=a_n\X^{n/q}+a_{n-1}\X^{n-1/q} + \ldots + a_1 \X^{1/q} + a_0 + a_{-1}\X^{-1/q}+a_{-2}\X^{-2/q}+\cdots\]
    where each $a_i\in \mc K$ and $n,q\in \mathbb{N}$, with $q>0$. It is well-known that $\mc K^{*}$ is a real closed field (see, e.g., Robert Walker's monograph \cite{Walker1978}).\footnote{Most sources (including \cite{Walker1978}) formulate this result as: \emph{If $\mc K$ is algebraically closed, then so is the Puiseux extension $\mc K^{*}$ of $\mc K$.} As noted by Shepherdson \cite{shep:nons64}, by putting together this formulation with the fact that an ordered field $\mc F$ is real closed iff $\mc F[i]$ is algebraically closed (where $i^2=-1$), one can readily verify that the Puiseux extension of a real closed field is real closed.}
    
    Given $k^*$ as above, let  \[[k^{*}]=a_n\X^{n/q}+a_{n-1}\X^{n-1/q} + \ldots + a_1 \X^{1/q} + [a_0],\] where $[a_0]$ is the unique element of $\mc R$ of distance less than 1 from $a_0$ such that $[a_0]\leq a_0$; note that $[a_0]$ is well-defined since $\mc R$ is an integer part of $\mc K$. 
        Observe that $[k^{*}]\in \shef (\mc R,\X)$ and $$0< k^*<[k^*]+1,$$ which makes it clear that $\shef (\mc R,\X)$ is an integer part of $\mc K^{*}$. Thus by Theorem \ref{Shepherdson_thm}, $\mc R^{\mathsmaller{\geq 0}}\models {\sf IOpen}$. 
\end{proof}
Recall from Section \ref{smurferella} that $\hat{\mathbb{A}}= \mathbb{Z}[\X_{q}\mathop\mid q\in \mathbb{Q}]$, and for $\fsvar \in [\mathbb{Q}]^{<\omega}$, $\hat{\mathbb{A}}(\fsvar)$ is the subring of $\hat{\mathbb{A}}$ generated by $\{\X_q \mathop\mid q\in \fsvar\}$.
\begin{definition} \label{def_of F, B-hat(r)}
   For $\fsvar \in [\mathbb{Q}]^{<\omega}$, the ordered commutative ring $\hat{\mathbb{B}}(\fsvar)$ is defined as follows: 

\begin{itemize}
\item
    $\hat{\mathbb{B}}({\varnothing}) := \mathbb{Z}$,
        \item
    $\hat{\mathbb{B}}({\fsvar\cup \verz q}) : = \shef(\hat{\mathbb{B}}(\fsvar), \X_{q})$, where $q>q'$ for each $q'\in \fsvar$.
\end{itemize}
\end{definition}
\begin{lem}\label{lem_shep_ind}
    Suppose $\fsvar\in [\mathbb{Q}]^{<\omega}$ and $q\in \mathbb{Q}$ such that $q>q'$ for all $q'\in \fsvar$. Then we have:
    $$\hat{\mathbb{B}}({\fsvar\cup \verz q})\cong \shef(\hat{\mathbb {B}} (\fsvar) ,\X_q).$$
    Moreover, the isomorphism is unique provided that it fixed each $\X_{q'}$ for all $q' \in \fsvar$.
\end{lem}

\begin{proof}
We prove this by an easy induction on the cardinality of $\fsvar$ using Lemma~\ref{kleinesmurf}.
\end{proof}

\begin{cor} \label{flatting_B(r)}
    For every $\fsvar\in [\mathbb{Q}]^{<\omega}$, there is an order preserving ring embedding 
    $f: \hat{\mathbb{B}}(\fsvar)\hookrightarrow \mathsf{rc}\hat{(\mathbb{A})}$. Moreover, such an embedding is unique if it fixes each $\X_q$ for $q\in \fsvar$.
\end{cor} 

\begin{rem} \label{directedness_of_B-hat(r)}
    
In light of Corollary \ref{flatting_B(r)}, from now on we 
identify each ring $\hat{\mathbb{B}}(\fsvar)$ for $\fsvar\in[\mathbb{Q}]^{<\omega}$ with the corresponding 
subring of $\mathsf{rc}\hat{(\mathbb{A})}$. Moreover, note that 
$$q\in \fsvar\textit{~iff~}{\sf X}_q\in \hat{\mathbb{B}}(\fsvar).$$

Moreover, as noted in Remark \ref{B(Y) is monotone}, for $\fsvar,\fsvar' \in [\mathbb{Q}]^{<\omega}$, we have:$$\fsvar\subseteq \fsvar'
\textit{~iff~} \hat{\mathbb{B}}(\fsvar)\textit{~is~a~subring~of~} \hat{\mathbb{B}}(\fsvar').$$
\end{rem}

\begin{definition}\label{B-definition}
$\mathbb{B}$ (the $\mathbb{Q}$-iterated Shepherdson model) is defined as $\hat{\mathbb{B}}^{\mathsmaller{\geq 0}}$, where $$\hat{\mathbb{B}}:=\bigcup\limits_{\fsvar\in [\mathbb{Q}]^{<\omega}}\hat{\mathbb{B}}(\fsvar).$$

The \emph{support} of an element $b\in\hat{\mathbb{B}}$, written $\mathsf{supp}(b)$, is the smallest 
$\fsvar\in[\mathbb{Q}]^{<\omega}$ (in the sense of inclusion) such that $b\in \hat{\mathbb{B}}(\fsvar)$. 
As shown in Appendix \ref{Section_on_supports}, $\mathsf{supp}(b)$ is well-defined. 
\end{definition}

\begin{cor}  \label{flatting}There is an order preserving ring embedding $F:\hat{\mathbb{B}}\hookrightarrow \mathsf{rc}\hat{(\mathbb{A})}$. Moreover, such an embedding is unique if it fixes each $\X_q$ for $q\in \mathbb{Q}$.   
\end{cor}

\begin{proof}
Enumerate $\hat{\mathbb{B}}$ as a countable sequence, and recursively construct $F$ by successive applications of Corollary \ref{flatting_B(r)}.
\end{proof}

\begin{rem} \label{rc(A)=rc(B)}In light of Corollary \ref{flatting} we shall identify $\hat{\mathbb{B}}$ with the corresponding subring of $\mathsf{rc}\hat{(\mathbb{A})}$. Thus $\hat{\mathbb{A}}\subseteq \hat{\mathbb{B}} \subseteq \mathsf{rc}(\hat{\mathbb{A}})$, which makes it clear that $\mathsf{rc}\hat{(\mathbb{A})}=\mathsf{rc}\hat{(\mathbb{B})}.$
\end{rem}

\begin{rem} \label{directness_of_B(r)}Let $~\mathbb B(\fsvar) :=\hat{\mathbb{B}}(\fsvar)^{\mathsmaller{\geq 0}}$. Note that $\mathbb{B}=\bigcup\limits_{\fsvar\in [\mathbb{Q}]^{<\omega}}\mathbb{B}(\fsvar).$ 
Coupled with Remark \ref{directedness_of_B-hat(r)}, this implies that for $\fsvar,\fsvar' \in [\mathbb{Q}]^{<\omega}$, we have:
$$\fsvar\subseteq \fsvar'\textit{~iff~} \mathbb{B}(\fsvar)\textit{~is~a~submodel~of~} \mathbb{B}(\fsvar').$$
\end{rem}

\begin{lem} \label{B_directed_union}$\mathbb{B}$ is the directed union of its constituent models $\mathbb{B}(\fsvar)$, 
as $\fsvar$ ranges in $[\mathbb{Q}]^{<\omega}.$ 
\end{lem}

\begin{proof}This is an immediate consequence of Remark \ref{directness_of_B(r)}.
\end{proof}

\begin{lem}\label{B(r)modelsIOpen}
   $\mathbb{B}(\fsvar)\models {\sf IOpen}$ for every $\fsvar\in[\mathbb{Q}]^{<\omega}.$
\end{lem}

\begin{proof}This follows from $|\fsvar|$ applications of Theorem \ref{Marker_thm} and Lemma \ref{lem_shep_ind}.
\end{proof}

\begin{theorem}\label{B_models_IOpen}
    $\QS\models {\sf IOpen}$.
\end{theorem}

\begin{proof}This follows by putting Lemma \ref{B_directed_union}, and Lemma \ref{B(r)modelsIOpen} together with the fact that ${\sf IOpen}$ is $\forall\exists$-axiomatizable and $\forall\exists$-sentences are preserved under directed unions.
\end{proof}

\begin{rem} \label{integer_parts of_F_and_F(Y)}Note that in light of \textup(Shepherdson's\textup) 
Theorem \ref{Shepherdson_thm}, Theorem \ref{B(r)modelsIOpen} implies that $\hat{\mathbb{B}}(Y)$ is an integer part of $\mathsf{rc}(\hat{\mathbb{A}}(Y))$ for every 
$\fsvar\in[\mathbb{Q}]^{<\omega}$, and Theorem \ref{B_models_IOpen} implies that $\hat{\mathbb{B}}$ is an integer part of $\mathsf{rc}(\hat{\mathbb{A}})$. 
This observation comes handy in the proof of Theorem \ref{B(Y) in rc(B(Y))} in the Appendix.
\end{rem}

\begin{definition}Recall (from Section \ref{smurferella}) that for $q \in \mathbb{Q}$, $\mathbb{Q}_q:=\{q'\in \mathbb{Q}\mathop\mid q'<q\}$. We put
 $$\mathbb{B}_{q}:=\bigcup\limits_{\fsvar\in [\mathbb{Q}_{q}]^{<\omega}}\mathbb{B}(\fsvar).$$ 
\end{definition}

It can be readily checked that $\mathbb{B}_{q}\subsetneq_{\mathrm{end}} \mathbb{B}_{q'}$ whenever $q<q'$. Indeed, as shown in Theorem \ref{eee_for_B}, $\mathbb{B}_{q}\prec_{\mathrm{end}} \mathbb{B}$ for each $q\in \mathbb{Q}$, which makes it clear that $\mathbb{B}_{q}\prec_{\mathrm{end}} \mathbb{B}_{q'}$ for $q<q'$.

Recall that $\shef(\mc R, \X)$ is the Shepherdson extension of $\mc R$ that adjoins an indeterminate $\X$ to $\mc R$. Note that for $q<q'$ we have: $$\mathbb{B}_q\subsetneq_{\mathrm{end}} \shef(\hat{\mathbb{B}}_q, \X_q)^{\mathsmaller{\geq 0}}\subsetneq_{\mathrm{end}} \mathbb{B}_{q'}.$$ 

We also observe that $\{\X_{q}^{n} \mathop\mid n\in \mathbb{N}\}$ is cofinal in $\shef(\hat{\mathbb{B}}_q, \X_q)^{\mathsmaller{\geq 0}}$, and 
$\{\X_{q}^{1/n} \mathop\mid n\in \mathbb{N}\}$ is downward cofinal in $\shef(\hat{\mathbb{B}}_q, \X_q)^{\mathsmaller{\geq 0}} \setminus \mathbb{B}_q.$

\begin{theorem}\label{eee_for_B}
   The model $\mathbb B_q$ is a downwards closed cut of $\mathbb B$ which is closed under the arithmetical
   operations. Moreover, $\mathbb B_q \prec_{\mathrm{end}} \mathbb{B}$.
\end{theorem}
\begin{proof} 
The proof is similar to the proof of Theorem \ref{eee_for_A}, so it is left to the reader. 
 \end{proof}
 
\begin{theorem}
    $\mathbb B$ satisfies {\sf Coll}.
\end{theorem}
\begin{proof} In light of Remark \ref{eee_and_coll}, this follows from Theorem \ref{eee_for_B}.
    \end{proof}

\begin{rem} The results of this section can be readily generalized as follows. Given any linearly ordered set $\mathbb{L}$ one can define the $\mathbb{L}$-iterated Shepherdson model $\mathbb{S}_{\mathbb{L}}$ by replacing $\mathbb{Q}$ with $\mathbb{L}$ in the definition of $\mathbb{B}$; thus $\mathbb{S}_{\mathbb{Q}} = \mathbb{B}$. $\mathbb{S}_{\mathbb{L}}$ satisfies $\mathsf{IOpen}$ for all linear orders $\mathbb{L}$. Moreover, $\mathbb{S}_{\mathbb{L}}$ satisfies $\mathsf{Coll}$ if $\mathbb{L}$ has the property that for any $\ell'\in \mathbb{L}$, there is some $\ell >_{\mathbb{L}} \ell'$ in $\mathbb{L}$ such that there is an order isomorphism $f: \mathbb{L}_{\ell} \rightarrow \mathbb{L}$ with the property that $f(x)=x$ for all $x<_{\mathbb{L}} \ell',$ where $\mathbb{L}_{\ell}$ is the initial segment of elements of $\mathbb{L}$ that are dominated by $\ell$. Of course $\mathbb{S}_{\mathbb{L}}$ also satisfies $\mathsf{Coll}$ if $\mathbb{L}$ is $\kappa$-like for some regular uncountable cardinal $\kappa$ (i.e., the cardinality of $\mathbb{L}$ is $\kappa$ but the cardinality of every proper initial segment of $\mathbb{L}$ is $\kappa$).
\end{rem}

\subsection{Further analysis of $\mathbb{B}$ } \label{subsection_on_further-analysis_of_B}
Let $\mathbb{F} = \rc(\hat{\mathbb{A}})$, and for each $q\in \mathbb{Q}$, let $\mathbb{F}_{q}:=\rc(\hat{\mathbb{A}}_q)$ (where $\mathbb{\hat{A}}$ and 
$\mathbb{\hat{A}}_{q}$ are as in Section \ref{smurferella}). The definitions and results of the previous section make it clear that 
$\hat{\mathbb{B}}$ can be described as the subring of $\mathbb{F}$ consisting of all the finite sums of the form $$a_n\X_{q_n}^{p_n}+\ldots+a_1\X_{q_1}^{p_1}+a_0,$$
where
\begin{enumerate}[$1$.]
    \item $q_i\in\mathbb{Q}$ and $p_i\in\mathbb{Q}^{\mathsmaller{> 0}}$, for $1\leq i \leq n;$
    \item $a_i \in \mathbb{F}_{q_i}$ for $1\leq i \leq n$;
    \item $a_0 \in \mathbb{Z}$.
\end{enumerate}
As before, the above representation is non-unique; for example, the expressions $\X_0^{2/4} + 0\cdot\X_0^{1/4}$ and $\X_0^{1/2}$ 
represent the same element. To obtain a unique representation, we introduce the notion of \emph{a standard form of an element of $\hat{\mathbb{B}}$} to be an expression of  form:
 
\begin{equation*}\label{QS_SF}
    a_n\X_{q_n}^{p_n} + a_{n-1}\X^{p_{n-1}}_{q_{n-1}} + \ldots + a_0,
\end{equation*}
where
\begin{enumerate}[$1$.]
\item 
$q_i\in \mathbb{Q}$, and $p_i\in \mathbb{Q}^{\mathsmaller{> 0}}$, for $1\leq i\leq n$;
\item 
$a_i\in \mathbb{F}_{q_i}$ and $a_i \neq 0$, for $1\leq i\leq n$;
\item 
for $n \geq i>j \geq 1$, we have $q_i\geq q_j$ and, if $q_i = q_j$, then $p_i>p_j$.
\end{enumerate}

To see why every element of $\mathbb{B}$ can be so presented, one simply ignores all the summands with coefficient $0$.

\begin{itemize}
    \item A standard form expression   $a_n\X_{q_n}^{p_n} + a_{n-1}\X^{p_{n-1}}_{q_{n-1}} + \ldots + a_0$ is the \emph{standard form for an element of $\mathbb{B}$} if $a_n\geq 0.$
\end{itemize}

The following definitions assume that $a\in \mathbb{B}$ is written in standard form.  
\begin{itemize}
    \item ${\sf rank}_0(a) = q_n$ if $n>0$ and $-\infty$ otherwise.
\item  ${\sf rank }(a) = (q_n,p_n)$, if $n>0$ and $-\infty$ otherwise.
\end{itemize}

\begin{rem}
Note that, for $x\in\mathbb{B}$, ${\sf rank}_0(x) := q$ iff $q$ is the least element of $$\{q'\in \mathbb{Q}  \mathop\mid x\in \shef(\hat{\mathbb{B}}_{q'},\X_{q'})\}.$$  Thus if ${\sf rank}_0(x) := q$, then $x$ can be written in the basic form $a\X_q^p+b$, where $p\in\mathbb{Q}^{\mathsmaller{> 0}}$, $a$ is a positive element of $\hat{\mathbb{B}}_{q}$, and $b$ is a Puiseux polynomial in $\shef(\hat{\mathbb{B}}_q, \X_q)$ of degree less than $p$. 
\end{rem}

\subsubsection{Definable generators of $\QS$}
The first two definitions are as in Section~\ref{smurferella}: 

\begin{itemize} 
    \item $x\ll y$ iff $\forall n\bin\omega\; (x+2)^n<y$.
    \item $x\approx y$ iff $\neg\, x\ll y \wedge \neg\, y \ll x$.\\
    \item $\m(x)$ iff  $\forall b\, \bigl(0<b\ll x\rightarrow b\mathop\mid x\bigr)$.
Here {\sf m} stands for `multiple'.
\end{itemize}
As in the case of the model $\mathbb{A}$ of the previous section, we have:
\begin{prop}
  ${\sf rank}_0(a)<{\sf rank}_0(b)$ iff $a\ll b$.  
\end{prop}

We say that $a\in \QS$ is a \textit{pseudo-variable} iff for some $q\in \mathbb{Q}$, $p\in\mathbb{Q}^{\mathsmaller{> 0}}$, and $b\in \QS_q^{\mathsmaller{> 0}}$, we have that
$b\X_q^p$ is the standard form of $a$. Crucially, for our purposes, this property is definable in $\QS$, as is shown in
the following theorem.

\begin{theorem}
    For every $a\in \QS$, $a$ is a pseudo-variable iff both of the following hold \textup(in $\QS$\textup):
    \begin{enumerate}[$1$.]
        \item $\m(a)$;
        \item $\forall b,c\, \bigl((a=bc \wedge a\approx b\approx c)\rightarrow (2\mid b\wedge 2\mid c)\bigr).$  
    \end{enumerate}
\end{theorem}

\begin{proof}
    Fix $a$ and assume that $a$ is written in the standard form.
    Assume first that $a$ is a pseudo-variable, hence $a = a_1\X^{p}_q$ for $a_1\in \QS_q^{\mathsmaller{> 0}}$. We first show that
    $a$ satisfies $(1)$.  Suppose 
    $0< b\ll a$. Then, both  $b$ and $a_1$ are in $\QS_q$. Since $b$ is nonzero, 
    $a_1b^{-1}\X_q^p\in \shef(\mathbb{B}_q, \X_q)\subseteq \mathbb{B}$. 
    So, $b$ divides $a$. 
    
    We show that $a$ satisfies $(2)$. We use the following convention: the last summmand in the standard form of an element $b\in\mathbb{B}$ will be denoted $b^*$. Assume that $a = bc$ and $a\approx b\approx c$. Aiming at a contradiction assume that $b$ is odd, i.e., there is $d$ such that $b=2d+1$ (here we are using the fact that in $\iopen$ every number is odd or even, see \cite[Chapter I, Lemma 1.17]{haje:meta91}). Note that this means that $b_0$ (the unique summand in the standard form not multiplied by a variable) is odd and $b^*=b_0$. Since $a\approx b\approx c$ and $a=bc$, it follows that ${\sf rank}(a)>{\sf rank}(b)$ and ${\sf rank}(a)>{\sf rank}(c)$. By the definition of the standard form, the last summand in the standard form of $bc$ is $c^*(b_i\X^{p_i}_{q_i} + \ldots +b_1\X^{p_1}_{q_1} +b_0)$, where $i$ is the greatest index in the standard form of $b$ such that ${\sf rank}_0(c^*)>q_i$ (we allow $i=0$ and hence $(bc)^* = b_0c^*$). However, ${\sf rank}(|(bc)^*|) = {\sf rank}(|c^*|) \leq {\sf rank}(c)<(q,p).$ Let $r$ be such that ${\sf rank}(c^*)<(q,r)<(q,p)$. This means that $a$ is not divisible by $\X^r_q$ - a contradiction.

    Now assume $a$ satisfies $(1)$ and $(2)$. We observe that by 2. $a$ must be nonstandard. Let $a_n\X_q^{p_n} + \ldots + a_1\X_q^{p_1} + a'$ 
    be the standard form of $a$ and ${\sf rank}_0(a')<q$.  If $a'\neq 0$, then, for some $q'<q$, the element $a$ is not divisible by $\X_{q'}$, 
    contradicting $\m(a)$. So, $a' = 0$. We claim that also $a_{n-1}, \ldots, a_1$ are all equal to $0$. 
    Assume not and let $i$ be the least such that $a_i\neq 0$. Then
    \[a=a_n\X_q^{p_n} + \ldots + a_i\X_q^{p_i} = |a_i|\X_q^{p_i}(a_n'\X_q^{p_n'} + \ldots +a_{i+1}'\X_q^{p_{i+1}'} +a_i/|a_i|),\]
    where $p_n',\ldots, p_{i+1}'$ are in all in $\mathbb{Q}^{\mathsmaller{> 0}}$. Note that $a_i/|a_i|\in\{1,-1\}$.
    However, $$a_n'\X_q^{p_n'} + \ldots +a_{i+1}\X_q^{p_{i+1}'} +a_i/|a_i|$$ is odd, 
    contradicting that $a$ satisfies $(2)$.
\end{proof}

\begin{cor}
    The Dorroh witnessing ideal $J$ of $\QS$ is $\QS$-definable without parameters.
\end{cor}
\begin{proof}
Put $J(x)$ to be the formula ``there is a sequence of pseudo-variables whose length is in $\omega$ such that $x$ is the sum of this sequence''.
\end{proof}

\subsubsection{Automorphisms} 

In this subsection, we develop the basics of automorphisms of $\mathbb{B}$.

\begin{theorem}\label{lem_aut_IOpen+Coll}
    For any two sequences of pseudo-variables $x_0\ll x_1\ll \ldots \ll x_{\ell-1}$, 
    $y_0\ll y_1\ll \ldots \ll y_{\ell-1}$ 
    there is an automorphism $\sigma$ of $\QS$ which moves $x_i$ to $y_i$. Moreover, 
    this automorphism extends to an automorphism of $\hat {\mathbb B}$;
    indeed the automorphism can be extended to $\mathsf{rc}\hat{(\mathbb{B})}.$
\end{theorem}
\begin{proof}
    Clearly, as in the proof of Theorem \ref{autosmurf}, we can assume that $x_i = \X_{q_{i}}$ and 
    $y_i = r_i\X_{q_i}^{p_i}$, where $q_i<q_j$ if $i<j<\ell $. 
    We note that,
    by the definition of $\QS$, 
    the $r_i$ are positive. Let $\{q'_0,\ldots, q'_{k-1}\} = \fsvard\subseteq \mathbb{Q}$ be a finite set
    such that all $r_i\X_{q_i}^{p_i}$ are in $\mathbb{B}(\fsvard)$, where 
    $q'_i < q'_j$, for $i<j<k$. We note that the $q_i$ will be in $Q$.
    For $q' \in \fsvard$, we set:
    \begin{itemize}
        \item 
    $\nu(\X_{q'}) := r_i \X_{q'}^{p_i}$ if $q'=q_i$ for $i< \ell$, 
    and  $\nu(\X_{q'}) := \X_{q'}$, otherwise.
    \end{itemize}
    
    For $k'\leq k$, 
    let $\fsvard_{k'} = \{q'_0,\ldots, q'_{k'-1}\}$.  By induction on $k'\leq k$, we show that 
    there is a unique automorphism $\sigma_{q'_{k'}}$ of $\mathbb{B}({\fsvard_{k'}})$ which 
    extends $\nu$.
    
    Assume that our claim holds for $k'$ and set $q' = q'_{k'}$. Let $\sigma_{\fsvard_{k'}}$ be the 
    automorphism of $\mathbb{B}({\fsvard_{k'}})$ 
    given by the inductive assumption. Let $\tau$ be the unique extension of $\sigma_{\fsvard_{k'}}$ to 
    $\rc(\mathbb{B}(\fsvard_{k'}))$. Finally, we define $\sigma_{\fsvard_{k'+1}}$. Given an element $z$ of the form
    \[a_n\X_{q'}^{n/m}+a_{n-1}\X_{q'}^{(n-1)/m} + \ldots + a_1 \X_{q'}^{1/m} + a_0,\]
    put $\sigma_{\fsvard_{k'+1}}(z)$ to be equal to
    \[\tau(a_n)\nu(\X_{q'})^{n/m} + \tau(a_{n-1})\nu(\X_{q'})^{(n-1)/m} + 
    \ldots + \tau(a_1)\nu(\X_{q'})^{1/m} + \tau(a_0).\]
        The function we defined clearly preserves addition and multiplication. 
   
    We check that $\sigma_{\fsvar_{n+1}}$ is a bijection: we define its 
    inverse, say $\rho$. We first define the auxiliary function $\mu$. Let $q \in \fsvard$.
    \begin{itemize}
        \item 
    $\mu(\X_{q'}) = \tau^{-1}(r_i)^{-\frac{1}{p_i}}\X_{q_i}^{\frac{1}{p_i}}$ if $q' = q_i$ for $i< \ell$, and
    $\mu(\X_{q'}) = \X_{q'}$, otherwise.
    \end{itemize}
    For an element $w$ in the form
    \[b_n\X_{q'}^{n/m}+b_{n-1}\X_{q'}^{(n-1)/m} + \ldots + b_1 \X_{q'}^{1/m} + b_0,\]
    we define $\rho(w)$ to be {\small
    \[\tau^{-1}(b_n)\mu(\X_{q'})^{n/m} + \tau^{-1}(b_{n-1})\mu(\X_{q'})^{(n-1)/m} + 
    \ldots + \tau^{-1}(b_1)\mu(\X_{q'})^{1/m} + \tau^{-1}(b_0),\]
    }
    
     We check that $\rho$ is indeed the inverse of $\sigma_{\fsvar_{n+1}}$, for the case that $q'=q_i$. The other
     case is simpler. Since $\sigma_{\fsvar_{n+1}}$ and $\rho$ clearly commute with addition, it suffices
     to do the computation for a single summand $a_j\X_{q_i}^{j/m}$ or $b_j\X_{q_i}^{j/m}$.
     We have:   
\begin{eqnarray*}
\rho \sigma_{\fsvar_{n+1}}(a_j\X_{q_i}^{j/m}) & = & 
\rho(\tau(a_j)(r_i\X_{q_i}^{p_i})^{j/m}) \\
 & = &  \tau^{-1}(\tau(a_j)r_i^{j/m})((\tau^{-1}(r_i)^{-\frac{1}{p_i}}\X_{q_i}^{\frac{1}{p_i}})^{p_i})^{j/m} \\
 & = &  a_j\tau^{-1}(r_i^{j/m})\tau^{-1}(r_i)^{-j/m}X_{q_i}^{j/m} \\
 & = & a_j\X^{j/m}_{q_i}\\
 \sigma_{\fsvar_{n+1}}\rho( b_j\X_{q_i}^{j/m}) & = & 
 \sigma_{\fsvar_{n+1}}( \tau^{-1}(b_j)(\tau^{-1}(r_i)^{-\frac{1}{p_i}}\X_{q_i}^{\frac{1}{p_i}})^{j/m})\\
 & = &  \tau(\tau^{-1}(b_j)(\tau^{-1}(r_i)^{-\frac{1}{p_i}})^{j/m})
 ((r_i\X_{q_i}^{p_i})^{\frac{1}{p_i}})^{j/m} \\
 &=&  b_j(r_i^{-\frac{1}{p_i}})^{j/m}(r_i^{\frac{1}{p_i}})^{j/m}\X_{q_i}^{j/m}\\
 &=& b_j\X_{q_i}^{j/m}
\end{eqnarray*}

    Clearly, the automorphism thus defined is unique.

    Since for each $\fsvard\in[\mathbb{Q}]^{<\omega}$, the function $\sigma_\fsvard$ is unique, we obtain that if $\fsvard\in[\mathbb{Q}]^{<\omega}$ contains all the $r_i\X_i^{q_i}$'s and 
    $\fsvard\subseteq \fsvard'\in[\mathbb{Q}]^{<\omega}$, 
    then $\sigma_\fsvard\subseteq \sigma_{\fsvard'}$. Hence,
    by the directedness of $\{\mathbb{B}(\fsvard): \mathbb{Q}^{<\omega}\}$, we can define an automorphism $\sigma$ of 
    $\mathbb{B}$ as the union of $\sigma_\fsvard$'s, for all sufficiently large $\fsvard$'s.

    We prove that $\sigma$ generates also an automorphism $\widetilde{\sigma}$ of $\mathsf{rc}\hat{(\mathbb{B})}$. Recall from Subsection \ref{subsection_on_further-analysis_of_B} that $\mathbb{F}:= \rc(\hat{\mathbb{A}})$.   Take any $\alpha\in\mathbb{F} $.
    Then, there is $\fsvar\in[\mathbb{Q}]^{<\omega}$, 
    such that $\alpha\in\rc(\hat{\mathbb{A}}(\fsvar))$. Obviously $\alpha\in \rc(\hat{\mathbb{B}}({\fsvar}))$ 
    and we know that there is a unique extension $\sigma_1$ of $\sigma_{\fsvar}$ to $\rc(\hat{\mathbb{B}}(\fsvar))$. 
    We put $\widetilde{\sigma}(\alpha) = \sigma_1(\alpha)$.
\end{proof}

\subsection{Interpreting $\mathbb{N}$ in $\mathbb{B}$}
It suffices to show that $\hat{\mathbb{B}}$ is a Dorroh ring, which is 
readily seen from the definition of the elements of $\mathbb{B}$: if $J$ is the ideal generated by all the elements of the form $\{\X_q^p : q\in\mathbb{Q}, p\in \mathbb{Q}^{\mathsmaller{> 0}}\}$, then $\mathbb{B}/J$ is isomorphic to $\mathbb{Z}.$ We let $\omega(x)$ be the $\Delta_0$ arithmetical formula defining $\mathbb{N}$ in $\mathbb{B}$. We call the resulting interpretation {\sf N}.

\subsection{Interpreting $\QS$ in $\mathbb{N}$}
We use the coding from Section \ref{realsmurf} to represent arithmetically the structure of $\QS$: an element of $\QS$ 
is represented via its standard form as a sequence of triples 
$\{\langle a_i, q_i, p_i\rangle\}_{i\leq n}$ where,  for each $i$, 
$a_i\in \mathbb{F}_{q_i}$,
 $q_i$'s are rational numbers (representing the indices of variables) and $p_i$'s are non-negative rational numbers (representing the exponents; we assume that $q_0 = 0$). We note that $a_i$'s are represented as in Section \ref{realsmurf}, so each $a_i$ is a pair $(p,l)$, where $p$ is a polynomial with coefficients in $\mathbb{Z}[\X_q:q<q_i] (= \hat{\mathbb{A}}_{q_i})$.  We assume that the sequences of $a_i$'s, $q_i$'s and $p_i$'s meet the requirements defining the standard form of an element of $\mathbb{B}$. 

We shall often view the sequence of triples $\{\langle a_i, q_i, p_i\rangle\}_{i\leq n}$ as the triple of sequences $\langle \bar{a},\bar{q},\bar{p}\rangle$. For $q\in\mathbb{Q}$, let 
$\sim_q$ abbreviate $$\sim_{\mathbb{Z}[\{\X_r:r<q\}]},$$and let 
$\sim$ abbreviate $$\sim_{\mathbb{Z}[\{\X_r:r\in\mathbb{Q}\}]},$$ 
where for a ring $\mc R$, $\sim_{\mc R}$ is defined in Section \ref{realsmurf}.  We say that sequences $s=\langle \bar{a}_1,\bar{p}_1,\bar{q}_1\rangle$, and $t=\langle \bar{a}_2, \bar{p}_2, \bar{q}_2\rangle$ represent the same element of $\mathbb{B}$ iff ${\sf length}(s) = {\sf length}(t)$ and for each $i<{\sf length}(s)$, $\bar{a}_1(i)\sim \bar{a}_2(i)$, $\bar{p}_1(i) = \bar{p}_2(i)$, $\bar{q}_1(i) = \bar{q}_2(i)$. By the results of Section \ref{realsmurf} such a presentation gives rise to a recursive ring, which is isomorphic to $\mathbb{B}$.

As previously, we will use $\mathbb{B}^{\alpha}$ to denote that $\mathbb{N}$-internal representation of $\mathbb{B}$.
We call the resulting interpretation {\sf B}.

 \subsection{ Back-and-forth in $\mathbb{B}$}
We start by defining the support function between
 elements of $\mathbb{B}$ and finite sets of variables. It is shown in Appendix \ref{Section_on_supports} (see Theorem \ref{properties_of_B-support} and the definition preceding it) that the support of an element $b\in \mathbb{B}$ (defined right after Definition \ref{B-definition}) coincides with the smallest set of variables $\mathfrak{X}$ such that $b\in \rc(\mathbb{Z}[\mathfrak{X}])$.
 
 Suppose that $b=a_{n}\X_{q_n}^{p_n} + a_{n-1}\X_{q_{n-1}}^{p_{n-1}} + \ldots + a_0$ is the standard form of an element from $\mathbb{B}$. A set $\mathfrak{X}$ is the support of $b$ iff $\mathfrak{X}$ is the least such that for all $i$, $\X_{q_i}\in \mathfrak{X}$ and each $a_i$ is algebraic over $\hat{\mathbb{A}}(\mathfrak{X})$. Since each element $b\in \mathbb{B}^{\alpha}$ is a polynomial, whose coefficients are given by some concrete notations, we shall take ${\sf supp}^{\alpha}(b)$ to be the set of variables that occur in $b$ or the notations for coefficients of $b$. This is clearly not invariant under the equality relation on $\mathbb{B}^{\alpha}$, but this is not problematic for our purposes. 

Let us now define the appropriate variants of the {\sf ev} and {\sf app} functions from Section \ref{sect_repr_real_alg}. This is needed, since now we need to represent not only the real algebraic numbers, but also the elements of $\rc(\mathbb{Z}[\{\X_q : q\in\mathbb{Q}\}]).$ To this we will need: 
\begin{itemize}
    \item a code of a notation $(f,l)$ as in the Section \ref{realsmurf} for $\mc R = \mathbb{Z}[\X_q:q\in\mathbb{Q}]$, where $f$ is a polynomial 
    with coefficients in 
    $\mathbb{Z}[\X_q : q\in\mathbb{Q}]$. We shall represent $f$ as a sequence 
    $(f_1,\ldots,f_n)$, where each $f_i\in \mathbb{Z}[\X_q : q\in\mathbb{Q}]$.
    \item a finite tuple of variables $\bar{\X}:=\tupel{\X_{q_0},\ldots,\X_{q_{k-1}}}$, 
    where $k$ is the number of variables occurring in the $f_i$'s;
    \item a variable $\X_q$, where $q>q_i$.
\end{itemize} 

Having the above fixed, we can write the formula
\[Q_{\bar{\X},f,x}(\Y) := f_n[\bar{\X}] \Y^n + f_{n-1}[\bar{\X}]\Y^{n-1}x + \ldots + f_1[\bar{\X}]\Y x^{n-1} + f_0[\bar{\X}]x^{n},\]
where $f_i[\bar{\X}]$ results from $f_i$ by replacing the $i$-th variable with $\X_{q_i}$. The point of $Q_{\bar{\X},f,x}(\Y)$ is that $a$ is a root of $f[\bar{\X}]$ iff $a\X_q^{p}$ is a root of $Q_{\bar{\X}, f,\X_q^p}(\Y)$ in $\Y$ (assuming the sanity condition that $q$ is greater than all the indices in $\bar{\X}$). As in Section \ref{sect_repr_real_alg}, for a given $p$, we can actually turn this into parameter-free $\mathbb{B}$-definable function 
${\sf app}^*_{x,s}(f,y) = Q_{s,f,x}(y)$.
Now given the notation $(f,k)$, sequence $s$ of length $k$ and arbitrary $x$

\begin{itemize}
    \item 
    ${\sf ev}_{x,s}((f,k)) = z$ iff there is an $s'\in {\sf seq}$ such that:
    \begin{itemize}
    \item
    ${\sf length}(s')=k$,
        \item 
    $\pi(s',k-1) = z$,
    \item 
    for all $i<k$, we have ${\sf app}^*_{x,s}(f,\pi(s',i)) =0 $,
    \item
    for all $i<j<k$, we have $\pi(s',i) < \pi(s',j)$,
    \item 
    if $y \bleq z$ and 
    ${\sf app}^*_{x,s}(f,y) =0$, then, for some  $j<k$, we have $\pi(s',j)= y$.
    \end{itemize}
\end{itemize}
With these definitions, we can reprove a version of the extension lemma. We say that an element $a\in\mathbb{B}$ is \textit{real-algebraic} definable from a set $A$, if for some polynomial $f\in\mathbb{Z}[A][\Y]$ and some $k\in\mathbb{N}$, $a$ is the $k$-th root of $f$.

\begin{lem}\label{spelendesmurfIOPEN}
Let $W$ be a finite, independent set of \psvar. Then, there is a finite set of variables $V$ such 
that $W\cup V$ is independent and ${\sf supp}(W)$ is \textup(pointwise\textup) 
real-algebraic definable from $W\cup V$.
 \end{lem}

 \begin{proof}
     We borrow the notation from the proof of 
    Lemma \ref{spelendesmurf} and let $W^*$, $V$ and $Z$ be defined as 
    before.
    As previously, it suffices to show that the elements in $W^*$ are real-algebraic definable from $Z$. Let 
    $\X_{q_0}, \ldots, \X_{q_{k-1}}$ be all such elements. We use induction on $i\leq k-1$. 

    Suppose the claim holds for all $j<i$. Note that the $i$-th element of $W$ must be of the 
    form $a\X_{q_i}^p$. We know that $a\in \rc(\mathbb{Z}[\{\X_r\in {\sf supp}(W) : r<q_i\}])$, 
    so there is a notation $(f,l)$ and a sequence of variables $\bar{\X}\in {\sf supp}(W)$ with 
    indices less than $q_i$ such that ${\sf ev}_{\X_{q_i}^p, \bar{\X}}((f,l)) = a\X_{q_i}^p$ 
    and $\X_{q_i}$ is the unique element $y$ for which ${\sf ev}_{\X_{q_i}^p, \bar{\X}}((f,l)) = a\X_{q_i}^p$. 
    Unraveling the definition of ${\sf ev}$ this means that $\X_{q_i}^p$ is a root of
    \[f_n[\bar{\X}] (a\X_{q_i}^p)^n + f_{n-1}[\bar{\X}](a\X_{q_i}^p)^{n-1}\Y + \ldots + f_1[\bar{\X}](a\X_{q_i}^p)\Y^{n-1} + f_0[\bar{\X}]\Y^{n},\]
    seen as a polynomial in $\Y$. In particular, $\X_{q_i}$ is real algebraic definable over $Z\cup \bar{\X}$. 
    By inductive assumption $\bar{\X}$ is real-algebraic definable over $Z$, which concludes our proof.
 \end{proof}

 We shall keep the definition of a witnessing function unchanged, of course given the current meaning of ``being a pseudo-variable''. 
 We observe that given a representation $u\in\mathbb{B}^{\alpha}$ and a witnessing function $F$, we can uniformly define an element ${\sf val}_F(u)$. We note that the definition of ${\sf val}_F$ makes use of the ${\sf ev}$ function. We say that a partial function $G: \mathbb{B}^{\alpha}\rightarrow \mathbb{B}$ is good is there is a witnessing function $F$ such that for each $u\in{\sf dom}(G)$, $G(u) = {\sf val}_F(u).$

 \begin{theorem}\label{thm_extension lemma_IOP}
Suppose $W$ is an independent set of \psvar\ and $x$ is any element of $\mathbb{B}$. Then, there is a witnessing function $F$ such that ${\sf rg}(F)\supseteq W$ and for some $u\in\mathbb{B}^{\alpha}$, $x= {\sf val}_F(u)$.
\end{theorem}
\begin{proof}
    Let  $W$ and $x$ be as in the assumption. Let $\mathfrak{X}$ support $x$ and let $$\mathfrak{X}' = \mathfrak{X}\setminus \{\X_q : a\X_q^p\in W\}.$$ Let $V$ be as promised in Lemma~\ref{spelendesmurfIOPEN} for $W$. 
We take $U := W \cup V \cup \mathfrak{X}'$.  Since $V$ consists of variables,
$U$ is independent. Since ${\sf supp}(W)$ is real-algebraic definable from $W\cup V$,
 it follows that every element of $\mathfrak{X}$ is real algebraic definable from $U$. In particular, every coefficient of the standard form of $x$ is in the real closure of $\mathbb{Z}[U]$. Hence we can fix a sequence of notations $\{(f_i,n_i)\}_{i\leq k}$, rational numbers $\{p_i\}_{i\leq k}, \{q_i\}_{i\leq k}$ such that by putting $u = \tupel{\{(f_i,n_i)\}_{i\leq k}, \{q_i\}_{i\leq k}, \{p_i\}_{i\leq k}}$ we have $x = {\sf val}_F(u),$ where $F$ interprets the variables of $f_i$'s as the appropriate elements from $U$. 
\end{proof}

  With these definitions the Atomic, Backward and Forward Clauses of the Ehrenfeucht-Fra\"{\i}ss\'{e} game can be proved exactly as in Section \ref{sect_pam_ehren}. 
  
  The following question is prompted by the fact that $\mathbb{B}$ is a recursive model of $\mathsf{IOpen}$; see also Question \ref{IOpen_question}.
  \begin{question} \label{last_question}Is there a model of $\mathsf{IOpen + Coll}$ whose complete theory has a restricted axiomatization, and which satisfies normality and other desirable algebraic properties that are known to hold in some recursive (computable) nonstandard model of $\mathsf{IOpen}$?   
  \end{question}

\subsection{Completing {\sf IOpen} plus Collection}
We  apply Corollary~\ref{cor_mod_retract} with:
\begin{itemize}
    \item 
$\mathbb B$ in the role of $\mc M$, 
\item 
$\leq$ in the role of $\belowz$, 
\item 
$\mathbb N$ in the role of
$\mc N$, 
\item 
{\sf N} in the role of $\imath$,
\item 
{\sf B} in the role of $\jmath$. 
\item
$\omega_{\mbbs B}^{\sf c}$ in the role of $\theta$.
\item 
The finite statement of the atomic and back and forth clauses
for the good functions together with \pam\ and ${\pam}^\alpha$ delivers 
our desired witness of a ${\sf INT}^{\sf fw}_3$-retract.
\end{itemize}

It is easy to see that our choices realise the conditions of the Corollary \ref{cor_mod_retract}.
We note that $\tau_{\sf N}$ and $\omega_{\mbbs B}^{\sf c}$ are $\Delta_0$.
So, it follows that we have a sentence $\Phi$, and a set $\mc X$
of $\Delta_1$-formulas $X$  (over $\Phi$), such 
that $\Phi + \mc X$ axiomatises  $\mathsf{Th}(\mathbb{B})$.

We conclude:

\begin{theorem}\label{iopencollsmurf}
    The theory ${\sf IOpen}+{\sf Coll}$ has a completion that is axiomatised by a sentence $\Phi$ plus a set of $\Sigma_1$-sentences. 
\end{theorem}

\subsection{Free Riders}
As in the case of the Shepherdson model $\mathbb{S}$, there seem to be no free riders since the iterated Shepherdson model inherits the algebraic defects of $\mathbb{S}$.

\section{A Final Word} 
In this section we first collect the open questions that arose throughout the paper, and then we briefly describe the shape and content of our projected sequel(s) of this work.

\subsection{Open Problems}
\begin{question} \label{Q_about_IDelta_0}
Is there a consistent completion of $\mathrm{I}\Delta _{0}$ that is of restricted complexity?

As indicated in the discussion following Question \ref{motivatingQ} (in the introduction), new ideas and machinery seem to be needed to resolve the above question.
\end{question}
\begin{question}\label{Q_separation} Are there examples that show that the class $\mathcal{T}_{\mathrm{nrestr}}$ of arithmetical theories extending $\mathsf{Q}$ that 
do not have a restricted completion can be separated from the classes $\mathcal{T}_{\mathrm{MRDP}}$,    $\mathcal{T}_{\mathrm{truth}}$, and  
$\mathcal{T}_{\mathrm{nrec}}$ (defined towards the end of the introduction)?
\end{question}

\begin{question}
 Is there yet another way to construct a witnessing sentence for an ${\sf INT}_3^{\sf fw}$-Retract?

The above question was posed as  Question \ref{Q_about_retracts}. See subsection \ref{Sec_witness} for the motivation of this question.
\end{question}

\begin{question}
   Is there a characterization of the authomorphism group  $\mathrm{Aut}(\mathbb{A})$ of  $\mathbb{A}$? (a similar question can be asked about 
$\mathrm{Aut}(\mathbb{B})$).

The above question below was posed as Question \ref{Question_Aut(A)}. Note that $\mathrm{Aut}(\mathbb{A})$ contains a 
copy of the automorphism group $\mathrm{Aut}({\mathbb{Q}})$ of the linear order $\mathbb{Q}$, and therefore it is of size continuum. 
\end{question}

\begin{question}  
Is there a model of $\mathsf{IOpen + Coll}$ whose complete theory has a restricted axiomatization and which satisfies normality and other desirable algebraic properties that are known to hold in some recursive (computable) nonstandard model of $\mathsf{IOpen}$?   

The above question was posed as 
Question \ref{last_question}. We suspect that the answer is in the positive.
  \end{question}

\subsection{Future Work}
  Our work arose from the quest for a more balanced and nuanced understanding of the manifestations of incompleteness
  in the form of \emph{restricted complexity of the completions}, as in \cite{Enayat-Visser2024}. In contrast with the focus of this paper on weak arithmetical theories based on the $\Sigma_n$-hierarchy, in future work we aim to pursue the program of employing \emph{new measures
of complexity} to probe stronger theories. 

For example, let $\pa_n$ denotes the set of theorems of $\pa$ (Peano Arithmetic) of complexity at most $\Pi_n$, and consider the $\Delta_0(\Sigma_{n})$ hierarchy of arithmetical formulae, i.e., formulae that are $\Delta_0$ relative to some $\Sigma_{n}$-definable oracle. We have been able to establish the following result.
  
 \begin{theorem}\label{thm_pan}
    For every sufficiently large $n$ there is a $k$ such that, $\pa_n$ has a completion which is axiomatizable by a single sentence and an infinite set of $\Delta_0(\Sigma_{k})$-sentences. 
    \end{theorem}

    Note that in contrast to \iopen, no ${\mathsf{I}}\Sigma_n$ with nonzero $n$ can be axiomatized with a set of $\Sigma_n$ formulae since as shown in \cite{Enayat-Visser2024} no consistent completion of $\mathrm{I}\Delta_0+\mathsf{Exp}$ is of restricted complexity. We conjecture that actually $\pa_{n+2}$ admits a completion which is axiomatizable by a single sentence and a set of $\Delta_0(\Sigma_{n+1})$ sentences.

   We have also obtained an analogue of Theorem \ref{thm_pan} for $\mathsf{ZF}_n$ (the $\Pi_n$-consequences of $\mathsf{ZF}$-set theory).

In our future work we also aim to apply similar ideas to explore incompleteness phenomena for \emph{second order theories} of arithmetic and set theory.

\appendix

\section{Proof of Theorem~\ref{nijveresmurf}}\label{robuustesmurf}

We show the following in the Shepherdson model $\mathbb S$:
    if $\alpha$ is irreducible in $\mathbb{S}$, then $\alpha\in \mathbb{N}$.

We write $\sg(a) := a/|a|$, in case $a\neq 0$. In other words, $\sg(a)=1$, if $a>0$, and
${\sf sg}(a)=-1$ if $a<0$.

\begin{proof}
    Suppose $\alpha$ is non-standard. Say $\alpha = a_n\X^{n/q}+ \ldots +a_0$, where $n>0$.
    We show that $\alpha$ has a non-trivial factorisation.
    
    \skippy
    We first consider the case that $a_0=0$. Suppose $a_0= \ldots = a_m=0$, where $m<n$. 
    Then, 
     \[\alpha = \X^{1/(2q)} (a_n\X^{(2n-1)/(2q)}+ \ldots + a_{m+1}\X^{(2m+1)/(2q)}). \]
     
    \skippy
    We  treat the case that $a_0\neq 0$. We can safely assume that $a_0 = \pm1$, since otherwise 
    we can find a prime $p$ that divides $a_0$, which allows us to factor $\alpha$ as a 
    product of two nonunit elements of $\mathbb{S}$ by writing $p (\frac{1}{p}\alpha)$. 
    Consider $\beta := a_n\X^{3n}+ a_{n-1}\X^{3(n-1)} + \ldots + a_0$.
    
    Suppose $\beta=0$ has a real solution, say $b$. Then, 
    \[\beta = (\X -b) (c_{n-1}\X^{3n-1} +c_{n-2}\X^{3n-4} \ldots + c_0).\]
    We note that $-bc_0 = a_0$ and $c_{n-1}=a_n>0$.  We find:
    \[  \alpha = (|c_0|\X^{1/(3q)} +\sg(c_0)a_0)((c_{n-1}/|c_0|) \X^{(3n-1)/(3q)} + \ldots + \sg(c_0)).  \] %
  
    Clearly, both components are in $\mathbb S$ and unequal to $1$.
   
       Suppose $\beta=0$ has a complex, non-real solution, say $\mf d$. Then, $\overline {\mf d}$, the conjugate of $\mf d$, is also
    a solution which is unequal to $\mf d$. Let $(\X- \mf d)(\X - \overline{\mf d})$ be $\X^2+ d_1\X + d_0$, where the $d_i$ are real.
    Then,
    \[\beta = (\X^2+ d_1\X + d_0)(c_{n-2}\X^{3n-2} +c_{n-3}\X^{3n-5} \ldots + c_0).\]
    We note that $d_0c_0 = a_0$ and $c_{n-2}=a_n >0$.  We find:
    \[  \alpha = (|c_0|\X^{2/(3q)} +|c_0|d_1 \X^{1/(3q)} +\sg(c_0)a_0)((c_{n-2}/|c_0|) \X^{(3n-2)/(3q)} 
    + \ldots + \sg(c_0)).  \]
    Clearly, both components are in $\mathbb S$ and unequal to $1$.
\end{proof}

\begin{rem}
Of course, in most cases our decomposition is very inefficient.
For, example, if $a_0>1$, we have the simple decomposition
$\alpha = a_0 ((a_n/a_0)\X^{n/q}+ \ldots + 1)$. If
$a_0 < -1$, we have $\alpha = -a_0 ((-a_n/a_0)\X^{n/q}+ \ldots + ({-}1))$.
\end{rem}

\section{Where does Collection get its Bite?}\label{collectionsmurf}
In this appendix, we give a sharp proof of the well-known insight that collection implies
induction over a sufficiently strong arithmetical base theory. Specifically, we show that
we can choose this base theory to be  $\mathrm I \mathsf E_1$, to wit the theory of bounded existential induction.
We have:
 
 \begin{theorem}
 $\mathrm I \mathsf E_1 + \forall_1 \hyph {\sf Coll}\vdash \mathrm I\exists_1$.
 \end{theorem}
 
 The proof is a direct adaptation of the proof of \cite[Lemma 7.4, p84]{kaye:mode91}

\begin{proof}
We reason in $\mathrm I \mathsf E_1 + \forall_1 \hyph {\sf Coll}$. Let $\psi(x,\vv y)$ be open. We allow further parameters in $\psi$.
Consider any number $v$. Clearly,
$\forall x \bles v\,  \exists \vv y \, (\psi( x,\vv y) \vee \forall \vv z\, \neg\, \psi(x,\vv z))$. The formula
$ (\psi(x,\vv y) \vee \forall \vv z\, \neg\, \psi(x,\vv z))$ is, modulo provable equivalence, a $\forall_1$-formula.
So, by $\forall_1$-Collection, we have:
\[ \exists w\, \forall  x \bles v\,  \exists \vv y\bles w \, (\psi( x,\vv y) \vee \forall \vv z\, \neg\, \psi(x,\vv z)).\]
It follows that, for some $w$,
\[   \forall  x \bles v\, (\exists \vv y\, \psi( x,\vv y) \iff \exists \vv y\bles w\, \psi( x,\vv y)). \]

We assume $\exists \vv y\,\psi(0,\vv y)$ and $\forall x\, (\exists \vv y\,\psi(x,\vv y) \to \exists \vv y\,\psi({\sf s}x,\vv y))$.
We may conclude  that $0< v\to \exists \vv y\, \psi(0,\vv y)$ and \[\forall x\, ((x<v\to \exists \vv y\,\psi(x,\vv y)) \to ({\sf s}x< v\to \exists \vv y\,\psi({\sf s}x,\vv y))).\]
So, for some $w$, we have $0< v\to \exists \vv y\bles w\, \psi(0,\vv y)$ and 
 \[\forall x\, ((x<v\to \exists \vv y\bles w\,\psi(x,\vv y)) \to ({\sf s}x< v\to \exists \vv y\bles w\,\psi({\sf s}x,\vv y))).\]
By ${\sf E}_1$-induction, we find $\forall x\bles v\, \exists \vv y \bles w\, \psi(x,\vv y)$. Since $v$ was arbitrary, it follows that
$\forall x\, \exists \vv y \, \psi(x,\vv y)$. 
  \end{proof}
  
  \begin{theorem}
 $\mathrm I \mathsf E_1 +  {\sf Coll}\vdash {\sf PA}$.
 \end{theorem}
 
 \begin{proof}
 By the results in \cite{kaye:diop90}, we know that $\mathrm I \mathsf \exists_1$
 extends {\sf EA}, so:\qedright
 \begin{eqnarray*}
 \mathrm I \mathsf E_1 +  {\sf Coll} & \vdash & \mathrm I\exists_1 + {\sf Coll} \\
 & \vdash & {\sf EA} + {\sf Coll} \\
 & \vdash &  {\sf PA}
 \end{eqnarray*}
 \end{proof}

 \section{The support function}\label{Section_on_supports}

Recall that $\hat{\mathbb{A}}:=\mathbb{Z}[\,\mathsf{X}_{q}\mathop\mid q\in \mathbb{Q]}$%
, and $\hat{\mathbb{B}}$ is the $\mathbb{Q}$-iterated Shepherdson model, as
defined in Sections \ref{smurferella} and \ref{iterated_shepherdson} (respectively). In this subsection we develop the basic properties of the support function applied to the rings $\hat{%
\mathbb{A}}$ and $\hat{\mathbb{B}}$. This is straightforward for the case of 
$\hat{\mathbb{A}}$, but the case of $\hat{\mathbb{B}}$ takes more work,
which is to be expected, since $\hat{\mathbb{B}}$ is a far more complicated
structure. As we shall see, it will be useful to take advantage of an appropriate support
function for the field $\mathbb{F}:=\mathsf{rc}(\hat{\mathbb{A}})=%
\mathsf{rc}(\hat{\mathbb{B}})$ in order to develop the basic properties of
the support function in the ring $\hat{\mathbb{B}}.$

Given a linear order $\mathbb{L}$, all of the results in this section
readily generalize to the ring $\mathbb{Z}[\,\mathsf{X}_{\ell}\mathop\mid\ell\in \mathbb{L]%
}$, and to the $\mathbb{L}$-iterated Shepherdson model $\mathbb{S}_{\mathbb{L}}$, 
as here we only rely on the fact that $\mathbb{Q}$ is a linear order.

\subsection{The support function in $\hat{\mathbb{A}}$}
By definition, every element $a\in \hat{\mathbb{A}}$ either can be put in the form $m$
for some $m\in \mathbb{Z}$ (in which case $a$ is referred to as a \textit{%
constant polynomial}), or in the form $m+\pi _{1}+\cdot \cdot \cdot +\pi _{n}
$, where $m\in \mathbb{Z},$ and each $\pi _{i}$ is a \textit{neat monomial},
i.e., is a product of the form $b\mathsf{X}_{q_{1}}^{m_{1}}...\mathsf{X}%
_{q_{k}}^{m_{k}}$, where $b\in \mathbb{Z}$ is nonzero, $q_{i}\in \mathbb{Q}$%
, with $q_{1}>\cdot \cdot \cdot >q_{n}$, and $k$ and each $m_{i}$ are nonzero natural numbers.
Let us refer to such a representation of a nonconstant polynomial as a \textit{neat%
} \textit{monomial representation}.

\begin{lemma} Any two neat monomial representations of the same element of 
$\hat{\mathbb{A}}$ are equal up to the summing order of the neat monomials.
\end{lemma}
\begin{proof} If $e$ and $e'$ are neat monomial representations of the same element in $\hat{\mathbb{A}}$, then 

$(*)$ $e-e'=0$ holds in $\hat{\mathbb{A}}.$

If  $e'$ is not obtainable from $e$ by changing the summing order
of the neat monomials of $e$, then the set of neat monomials that occur in $e$ does not
coincide with the set of neat monomials that occur in $e'$. This shows
that the left-hand side of the equation $(*)$ is a nonzero polynomial, thus
contradicting $(*)$.
\end{proof}

The above lemma assures us that the definition below yields a well-defined function  $%
\mathsf{supp}:\hat{\mathbb{A}}\rightarrow \lbrack \mathbb{Q}]^{<\omega }$. We often refer to $\mathsf{supp}%
(a)$ as `the support of $a$'.

\begin{definition} Given $a\in \hat{\mathbb{A}},$ $\mathsf{supp}%
(a)=\varnothing $ if $a$ is a constant polynomial, and otherwise $\mathsf{%
supp}(a)$ is the set of $q\in \mathbb{Q}$ such that $\mathsf{X}_{q}$ occurs
in a neat monomial representation of $a$.
\end{definition}
Recall that for $Y\in \lbrack \mathbb{Q}]^{<\omega },$ $\hat{\mathbb{A}}(Y):=%
\mathbb{Z}[\,\mathsf{X}_{q}\mathop\mid q\in Y\mathbb{]}.$ The following characterization
of the support function can be readily established using the definitions
involved.

\begin{theorem} The following statements are equivalent for $a\in \hat{\mathbb{A}}$, and $Y\in \lbrack \mathbb{Q}%
]^{<\omega }$.
\begin{enumerate}[$1$.]
    \item  $\mathsf{supp}(a)=Y.$

 \item $Y$ is a minimal element of $\lbrack \mathbb{Q}]^{<\omega}$ such that $a\in \hat{\mathbb{A}}(Y)$.

 \item $Y$ is the minimum element of $\lbrack \mathbb{Q}]^{<\omega}$ such that $a\in \hat{\mathbb{A}}(Y)$.
\end{enumerate}
\end{theorem}
\subsection{The support function in $\hat{\mathbb{B}}$}
In order to develop the support function for $\hat{\mathbb{B}},$ we recall
some basic facts from field theory that allow us to get hold of an appropriate support
function for the field $\mathbb{F}:=\mathsf{rc}(\hat{\mathbb{A}})=%
\mathsf{rc}(\hat{\mathbb{B}})$.

Given any field $\mathcal{K}$, $\mathsf{acl}_{\mathcal{K}}:\mathcal{%
P}(\mathcal{K})\rightarrow \mathcal{P}(\mathcal{K})$ is the relative-to-$\mc K$ algebraic closure function  defined as follows: given  $S\subseteq 
\mathcal{K}$, $\mathsf{acl}_{\mathcal{K}}(S)=\{a\in \mathcal{K}:$ $a$ is
algebraic over $S\}$, where $a$ is said to be algebraic over $S$, if there
are $b_{1},\cdot \cdot \cdot ,b_{n}$ in $S$, and some nonzero polynomial $p(%
\mathsf{X}_{1},\cdot \cdot \cdot ,\mathsf{X}_{n},\mathsf{X}_{n+1})\in 
\mathbb{Q}[\mathsf{X}_{1},\cdot \cdot \cdot ,\mathsf{X}_{n},\mathsf{X}_{n+1}]
$ such that $p(b_{1},\cdot \cdot \cdot ,b_{n},a)=0$ holds in $\mathcal{K}$. It is a basic fact of field theory that $\mathsf{acl}_{\mathcal{K}}(S)$ is a subfield of $\mathcal{K}$ for all $S\subseteq \mathcal{K}$.

Under the above definition, $\mathsf{acl}_{\mathcal{K}}$ satisfies the
 properties given in Definition (\ref{olijkesmurf}) when $%
\mathsf{cl}$ is interpreted as $\mathsf{acl}_{\mathcal{K}}$, known as the axioms of \textit{pregeometry}\footnote{Pregeometries are also referred to as `finitely based matroids' in the literature. We found the chapter on Geometric Model Theory in the lecture notes by Manuel Bodirsky to be useful in the preparation of this section; see:\\
\hspace*{2cm}\tt{https://wwwpub.zih.tu-dresden.de/$\sim$bodirsky/Model-theory.pdf}}. The idempotence axiom in this context is often paraphrased by algebraists as the principle: \emph{algebraic over algebraic is algebraic}; and the exchange axiom is commonly known as the \emph{Steinitz Exchange Lemma}, which is a basic result in field theory. The rest of the axioms are simple consequences of the definitions involved.

\begin{definition}\label{olijkesmurf} Suppose $X$ is a set, and $\mathsf{cl}:\mathcal{P}(X)\rightarrow \mathcal{P}(X)$. $(X,\mathsf{cl})$ is pregeometry if the following statements hold for all subsets $S$ and $S'$ of $X$:
    
\begin{itemize}
\item Expansion: $S\subseteq \mathsf{cl}(S).$

\item Idempotence: $\mathsf{cl}(\mathsf{cl}(S))=\mathsf{cl}(S).$

\item Monotonicity: $S\subseteq S'\Rightarrow \mathsf{cl}%
(S)\subseteq \mathsf{cl}(S').$

\item Finite character: If $a\in \mathsf{cl}(S)$, then $a\in \mathsf{cl}(S')$,
for some finite $S'\subseteq S$.

\item Exchange axiom: If $a\in \mathsf{cl}(S\cup \{b\})$ and $a\notin 
\mathsf{cl}(S)$, then $b\in \mathsf{cl}(S\cup \{a\}).$
\end{itemize}

\end{definition}
Familiar facts from the elementary theory of vector spaces regarding the
notions of spanning, independence, basis, and dimension can be developed
within pregeometries (with similar proofs). More specifically, given
a pregeometry $(X,\mathsf{cl})$,  consider the following dictionary:
 
\begin{itemize}
\item 
For $S\subseteq X$, $\mathsf{span}(S):=\mathsf{cl}(S)$.
\item For $S\subseteq X$, $S$ is \textit{independent} if 
$a\notin \mathsf{cl}(S\backslash \{a\})$ for all $a\in S$.
\item A subset $B$ of $X$ is a \textit{basis} if $\mathsf{span}(S)=X$ and $S$ is independent.
\end{itemize}

The following encapsulates the basic results concerning the above notions.

\begin{theorem}\label{basic_thm_pregeometry}The statements below hold in every pregeometry $(X,\mathsf{cl}).$

\begin{enumerate}[$1$.]
\item  $S$ is a basis iff $\mathsf{span}(S)=X$ and there is no 
$S^{\prime}\subsetneq S$ such that $\mathsf{span}(S^{\prime})=X$.
\item  If $S$ is independent, then there is some basis $B$ such that 
$S\subseteq B.$
\item  If $\mathsf{span}(S)=X$, then there is some basis $B$ such that $%
B\subseteq S.$
\item  If $B$ and $B^{\prime }$ are bases, then $B$ and $B^{\prime }$ have
the same cardinality \textup(referred to as the \textit{dimension} of $(X,\mathsf{cl})$\textup).
\item  If $B$ is a basis and $a\in X$, then there is a unique minimal finite subset $Z$ 
of $B$ such that $a\in \mathsf{span}(Z)$.
\item If $Z$ and $Z'$ are subsets of a basis $B$, then $%
Z\subseteq Z'$ iff $\mathsf{span}(Z)\subseteq \mathsf{span}%
(Z')$.
\end{enumerate}
\end{theorem}
Note that if $B$ is a basis for a pregeometry $(X,\mathsf{cl})$, then part (5) of the above theorem assures us that we can meaningfully define the support function $$\mathsf{supp}_{(X,%
\mathsf{cl},B)}:\mathcal{P}(X)\rightarrow \lbrack B]^{<\omega },$$ by $\mathsf{
supp}_{(X,
\mathsf{cl},B)}(a)=Z$ if $Z$ is the unique minimal subset of $B$ such that $a\in \mathsf{span}(Z)$. 

The following result readily follows from Theorem \ref{basic_thm_pregeometry} and the definitions involved.
\begin{theorem}\label{support_in-pregeometry}Suppose $B$ is a basis for a pregeometry $(X,\mathsf{cl})$. The
following statements are equivalent for $a\in X$ and $Z\in \lbrack
B]^{<\omega }.$
\begin{enumerate}[$1$.]
\item 
$Z=\mathsf{supp}_{(X,\mathsf{cl},B)}(a)$.
\item $Z$ is the minimum element of $\lbrack
B]^{<\omega }$ such that $a\in \mathsf{cl}(Z).$
\end{enumerate}
\end{theorem}
Putting the fact that $(\mathcal{K},\mathsf{acl}_{\mathcal{K}})$ is a pregeometry together with the above, we obtain:

\begin{theorem} \label{support_in_field}Suppose $\mathcal{K}$ is a field, and $B$\ is a basis\footnote{
In this context $B$ is commonly referred to as a \textit{transcendence basis/base} for $\mathcal{K}$.} for the pregeometry  $(\mathcal{K},\mathsf{acl}_{\mathcal{K}})$.
The following statements are equivalent for $a\in \mc K$ and $Z\in  \lbrack B]^{<\omega }$:
\begin{enumerate}[$1$.]
   \item $Z$ is a  minimal element of $  \lbrack B]^{<\omega }$ such that $a\in \mathsf{acl}_{\mathcal{K}}(Z)$, in other words  
$Z=\mathsf{supp}_{(\mathcal{K},\mathsf{acl}_{\mathcal{K}},B)}(a)$. 

\item $Z$ is the minimum element of $ \lbrack B]^{<\omega }$ such that $a\in \mathsf{acl}_{\mathcal{K}}(Z).$
  

\end{enumerate}
\end{theorem}
This concludes our review of the relevant basic facts about pregeometries and
their relationship to fields. We will now apply the above machinery to the field 
$\mathbb{F}=\mathsf{rc}(\hat{\mathbb{A}})=\mathsf{rc}(\hat{\mathbb{B}})$ (see Remark \ref {rc(A)=rc(B)}).

\begin{itemize}
\item In
what follows, $\frak{X}:=\{\,\mathsf{X}_{q}\mathop\mid q\in \mathbb{Q}\}$,  and for 
$Y\in \lbrack \mathbb{Q}]^{<\omega }$, 
$\frak{X}_{Y}:=\{\,\mathsf{X}_{q}\mathop\mid q\in Y\}.$
\end{itemize}

\begin{theorem}\label{linking_thm}For every $a\in \mathbb{F}$ there is a unique $Y\in
\lbrack \mathbb{Q}]^{<\omega }$ such that $a\in \mathsf{rc}(\hat{\mathbb{B}}(Y))=\mathsf{acl}_{\mathbb{F}}(\frak{X}_{Y})$.
\end{theorem}

\begin{proof}$\frak{X}$ is clearly a basis for $(\mathbb{F},\mathsf{acl}_{\mathbb{F}})$, so we can invoke Theorem \ref{support_in_field}. 
Note that since $\mathbb{F}$ is real closed,  $\mathsf{acl}_{\mathbb{F}}(\frak{X}_{Y})=\mathsf{rc}(\hat{\mathbb{B}}(Y))$ for all 
$Y\in \lbrack \mathbb{Q}]^{<\omega }$.
\end{proof}

\begin{theorem} \label{B(Y) in rc(B(Y))} $\hat{\mathbb{B}}(Y)=\hat{\mathbb{B}}\cap \mathsf{rc}(\hat{\mathbb{B}}(Y))$ for each $Y\in \lbrack \mathbb{Q}]^{<\omega }.$
\end{theorem}

\begin{proof} The direction $\hat{\mathbb{B}}(Y)\subseteq \hat{\mathbb{B}}%
\cap \mathsf{rc}(\mathbb{B}(Y))$ is straightforward. The other direction follows from the 
fact that $\hat{\mathbb{B}}(Y)$ is an integer part of $\mathsf{rc}(\hat{\mathbb{B}}(Y))$ for 
each $Y\in \lbrack \mathbb{Q}]^{<\omega}$ (see Remark \ref{integer_parts of_F_and_F(Y)}).

More explicitly, suppose $b\in \hat{\mathbb{B}}\cap \mathsf{rc}(\hat{\mathbb{B}}(Y))$. 
Then, since $\hat{\mathbb{B}}(Y)$ is an integer part of $\mathsf{rc}(\hat{\mathbb{B}}(Y))$, 
there is some $b' \in \hat{\mathbb{B}}(Y)$ such that $\left\vert b-b'\right\vert <1$,
which implies $b=b'$, thus $b\in \hat{\mathbb{B}}(Y)$.
\end{proof}

Thanks to Theorem \ref{support_in_field}, we can meaningfully define the rank function for 
$b\in \hat{\mathbb{B}}$ as follows:
\begin{itemize}
\item   $\mathsf{supp}_{
\hat{\mathbb{B}}}(b)=\{q\in \mathbb{Q}\mathop\mid\mathsf{X}_{q}\in \mathsf{supp}_{(\mathbb{F},
\mathsf{acl},\frak{X}\}}(b)\}$.
\end{itemize}
Theorems \ref{linking_thm}, \ref{support_in_field}, and  \ref{B(Y) in rc(B(Y))} readily imply the following characterization of the support function in $\hat{\mathbb{B}}$.

\begin{theorem} \label{properties_of_B-support}The following are equivalent for $b\in \hat{\mathbb{B}}$ and $Y\in
\lbrack \mathbb{Q}]^{<\omega }.$
\begin{enumerate}[$1$.]
\item $Y=\mathsf{supp}_{\hat{\mathbb{B}}}(b)$.
\item   $Y$ is a minimal element of 
$\lbrack \mathbb{Q}]^{<\omega }$ such that $a\in \hat{\mathbb{B}}%
(Y).$

\item   $Y$ is the minimum element of 
$\lbrack \mathbb{Q}]^{<\omega }$ such that $a\in \hat{\mathbb{B}}%
(Y).$

\end{enumerate}
\end{theorem}
\begin{rem} \label{B(Y) is monotone}Using Theorems \ref{support_in_field} and \ref{B(Y) in rc(B(Y))} we can prove that $\mathbb{B}(Y)\subseteq \mathbb{B}(Y^{\prime })$
implies $Y\subseteq Y^{\prime }$ \textup(the converse is of course trivial\textup).  
\end{rem}
\end{document}